\documentclass[a4paper,11pt]{article}

\usepackage{amsfonts}
\usepackage{fancyhdr}
\usepackage{color}
\usepackage{graphicx}
\usepackage{dsfont}
\usepackage{extarrows}
\usepackage{mathtools}
\usepackage{amsthm}
\usepackage{tikz}
\usetikzlibrary{arrows}
\usepackage{hyperref}

\newcommand{\N}{{\mathbb{N}}}
\newcommand{\E}{{\mathbb{E}}}
\newcommand{\1}{{\mathds{1}}}
\newcommand{\PX}{{\mathbb{P}}}
\newcommand{\ER}{{Erd\H os-R\'enyi}}

\newcommand{\PA}{{preferential attachment}}
\newcommand{\Var}{{\mathrm{Var}}}

\theoremstyle{definition}
\newtheorem{dfn}{Definition}

\theoremstyle{plain}
\newtheorem{thm}[dfn]{Theorem}
\newtheorem{lemma}[dfn]{Lemma}

\numberwithin{dfn}{section}

\title{Random graphs and their subgraphs}
\author{Klemens Taglieber\\
		University of Stuttgart\\
		\and
		Uta Freiberg\\
		University of Stuttgart}
\date{\today}

\begin{document}

\maketitle


\begin{abstract} Random graphs are more and more used for modeling real world networks such as evolutionary networks of proteins. For this purpose we look at two different models and analyze how properties like connectedness and degree distributions are inherited by differently constructed subgraphs. We also give a formula for the variance of the degrees of fixed nodes in the \PA \ model and additionally draw a connection between weighted graphs and electrical networks.\end{abstract}


\section{Introduction}
The modeling and analysis of random graphs is a good possibility to understand and examine real world networks. The first random graphs were introduced between 1959 and 1961 by Paul Erd\H{o}s and Alfr\'ed R\'enyi \cite{ER59,ER60,ER61}. This model is connected to percolation theory which has several applications in physics. Up until today new models are developed such as the \PA \ model which was worked out by L\'aszl\'o Barab\'asi and R\'eka Albert \cite{BA99} in 1999. This model is suitable for the modeling of most networks we are surrounded by such as the Internet, the World Wide Web or friendship networks. More and more random graphs are used in biology to analyze a variety of mechanisms like for example the spreading of epidemics or the evolution of proteins. Looking at protein networks the question arises if one can predict the not yet discovered proteins or even how the network and therefore the proteins will evolve.\\
At first we want to look at random graphs and their subgraphs and their different properties. We especially focus on the degree distributions which are calculated and plotted using \texttt{R} \cite{R}. We will also compare subgraphs which are constructed  in different ways and determine if this yields to different subgraphs. We finally look at a protein network to find out if it can be constructed with one of the presented methods. The figures of graphs throughout the work are done with Gephi \cite{Gephi}, a software for visualizing graphs and networks.\\
As an interesting intermezzo we generalize the results from \cite{Tetali} to weighted graphs.\\
We start in section \ref{graphs} by giving some basic definitions of graphs and random walks on graphs, as well as proving some relationships between graph theory and the theory of electrical networks. In section \ref{models} we introduce two models for constructing random graphs and analyze these graphs and their subgraphs concerning their degree distributions. We then investigate a protein network in section \ref{proteins} on the possibility of modeling it and finally in section \ref{outlook} we give an overview on further possibilities to model such protein networks.


\section{Graphs}\label{graphs}
In this section we want to present some basic definitions and properties of graphs which are found in \cite{Bollobas}. We will also introduce Markov chains on graphs and analyze some of their properties. Finally we look at graphs by interpreting them as electric networks to generalize the results from \cite{Tetali}.

\begin{dfn}
A {\em graph} $G$ is a pair of disjoint sets $(V,E)$ where $E\subseteq V\times V$ is consisting of unordered pairs of elements of $V$. The elements in $V$ are called {\em nodes}, the elements of $E$ {\em edges}.
We call two nodes $x,y\in V$ {\em neighbors} if $\{x,y\}\in E$ and denote this with $x\sim y$.
\end{dfn}
We will only consider finite graphs here. A graph is called {\em finite} if $V$ is only finitely large, i.e. $\#V<\infty$. For a finite graph with $\#V=n\in \N$ we denote $V=[n]:=\{1,\ldots,n\}$.\\ Since in some cases we have more than one graph we will then denote $V$ with $V(G)$ and $E$ with $E(G)$.

\begin{dfn}
Let $G=(V,E)$ be a graph with $V=[n]$ for some $n\in\N$. Its {\em adjacency matrix} is then an $n\times n$-matrix $A=(a_{xy})_{x,y\in V}$ where
\[a_{xy}:=\1_{\{\{x,y\}\in E\}}:=\left\{\begin{array}{ll} 1, & \text{ if } \{x,y\}\in E, \\ 0, & \text{ otherwise.} \end{array}\right.\]
\end{dfn}

\begin{dfn}
Let $G=(V,E)$ be a graph. The {\em degree} of a node $x\in V$ is given by
\[D(x):=\sum_{y\in V}\1_{\{\{x,y\}\in E\}} =\sum_{y\in V : y\sim x}1=: \sum_{y\sim x}1.\]
\end{dfn}
By the definition of the adjacency matrix it follows directly for the degree of any node $x$ that
\[D(x)=\sum_{y\in V}a_{xy}.\]
\begin{dfn}
A {\em path} $P=(V(P),E(P))$ is a subgraph of $G=(V(G),E(G))$ with edge set $E(P)=\{\{x_0,x_1\},\ldots,\{x_{k-1},x_k\}\} \subseteq E(G)$ and vertex set $V(P)=\{x_0,\ldots,x_k\}\subseteq V(G)$ . The length of $P$ is given by the number of edges $k=\#E(P)$ it contains .
\end{dfn}
A path from $x$ to $y$ is a path with $x_0=x$ and $x_k=y$. Let $G=(V,E)$ be a graph and $x,y\in V$ we then say that $x$ and $y$ are in the same component of $G$, if there exists a path from $x$ to $y$.\\
If the graph $G$ only consists of one component, we call it {\em connected}.

We now generalize our definitions to weighted graphs.
\begin{dfn}
A {\em weighted graph} $G=(V,E,C)$ is a graph where we assign a weight $c_{xy}\in [0,\infty)$ to every pair $(x,y)\in V\times V$. We want the weights to be symmetric, hence $c_{xy}=c_{yx}$. The set of edges is then given by
\[E:=\{\{x,y\}\in V\times V: c_{xy}>0\}.\]
\end{dfn}
The weights $(c_{xy})_{x,y\in V}$, called conductances, give us analogously to the adjacency matrix the conductance matrix $C$ of the graph.
\begin{dfn}
Let $G=(V,E,C)$ be a weighted graph with edge weights $(c_{xy})_{x,y\in V}$, then the {\em conductance matrix} $C$ of $G$ is given by
\[C=(c_{xy})_{x,y\in V}.\]
\end{dfn}
Obviously the conductance matrix is symmetric. It is also possible to give an adjacency matrix for weighted graphs where $a_{xy}=\1_{\{c_{xy}>0\}}$ for all $x,y\in V$.
\begin{dfn}
Let $x\in V$ be a node of the weighted graph $G=(V,E,C)$ then the {\em generalized degree} of $x$ -- or {\em weight of node $x$} -- is given by
\[\mu_x:=\sum_{y\in V}c_{xy}=\sum_{y\sim x}c_{xy}.\]
\end{dfn}
If $c_{xy}\in\{0,1\}$ for all $x,y\in V$ the graph is not weighted and it holds $\mu_x=D(x)$ for all $x\in V$.

\subsection{Properties of graphs and Markov chains on graphs}
In the following we will consider a weighted graph $G=(V,E,C)$ with $V=[n], \ E\subseteq V\times V$, where $\#E:=m$, and conductance matrix $C$.
\begin{dfn}
For all $x\in V$ let $\mu_x$ be the generalized degree of $x$. The {\em degree distribution} of $G$ is then given by
\[\PX(\mu_x \geq s)=\frac{1}{n} \sum_{y\in V} \1_{\{\mu_y\geq s\}}.\]
\end{dfn}
We call the degree distribution of $G$ respectively the graph $G$ itself scale free if for $s\in [0,\infty)$, some $c_n\in (0,\infty)$ and some $\tau>1$ it holds that
\[\PX(\mu_x \geq s)\asymp c_ns^{-(\tau-1)}.\]
Equivalently we can look at the total number of nodes with degree $s$ or more denoted by $N_{\geq s}$ and get
\[N_{\geq s} \asymp \tilde c_ns^{-(\tau -1)} \text{ where } \tilde c_n\in (0,\infty).\]

For unweighted graphs it is sufficient to look at the number of nodes with exactly degree $k\in \N_0$ and the scale free property simplifies to $N_k\asymp \hat c_n k^{-\tau}$ for some $\hat c_n\in (0,\infty)$.
\begin{dfn}
Let $G=(V,E,C)$ be a weighted graph with conductance matrix $C$ and $(X_n)_{n\in \N_0}$ a homogeneous Markov chain with state space $V$. We then call $(X_n)_{n\in \N_0}$ {\em Markov chain on $G$} if its transition matrix $P=(p_{xy})_{x,y\in V}$ is for all $n\in \N,\ x,y\in V$ given by
\[p_{xy}=\PX(X_n=y|X_{n-1}=x):=\frac{c_{xy}}{\sum_{z\sim x}c_{xz}}=\frac{c_{xy}}{\mu_x}.\]
\end{dfn}
Since the initial distribution of the Markov chain is not important here we will not worry about it.

\begin{dfn}
The {\em hitting time of $y\in V$} for the Markov chain $(X_n)_{n\in \N_0}$ on $G$ is defined as 
\[\tau_y:=\inf\{n\geq 0: X_n=y\}.\]
\end{dfn}

\begin{dfn}
Let $G=(V,E,C)$ be a connected graph with $V=[n]$, $(X_n)_{n\in \N_0}$ a Markov chain on $G$ and $x\in V$. We then call
\[\E^x(\tau_y):=\E(\tau_y|X_0=x)\]
the {\em expected hitting time of $y$ with start in $x$}.
\end{dfn}
For every $x\in V$ it obviously holds that $\E^x(\tau_x)=0$.

\begin{lemma}\label{mvp}
For $x\not=y$ the expected hitting time of $y$ with start in $x$ satisfies
\[\E^x(\tau_y)=1+\sum_{z\sim x}p_{xz}\E^z(\tau_y).\]
\end{lemma}
\begin{proof} By using the Markov property for the fourth equality we get
\begin{eqnarray*}
\E^x(\tau_y)&=&\E(\tau_y\vert X_0=x)=\sum_{k\geq 1}k\PX(\tau_y=k\vert X_0=x)\\
&=&\sum_{k\geq 1}k\sum_{z\sim x}\frac{\PX(\tau_y=k, X_0=x, X_1=z)}{\PX(X_0=x)}\frac{\PX(X_0=x,X_1=z)}{\PX(X_0=x,X_1=z)}\\
&=&\sum_{k\geq 1}k\sum_{z\sim x}\PX(\tau_y=k\vert X_0=x, X_1=z)\PX(X_1=z\vert X_0=x)\\
&=&\sum_{z\sim x} p_{xz}\sum_{k\geq 1}k\PX(\tau_y=k-1\vert X_0=z)\\
&=&\sum_{z\sim x} p_{xz} \E^z(\tau_y+1)=1+\sum_{z\sim x} p_{xz}\E^z(\tau_y).
\end{eqnarray*}
\end{proof}
With the boundary condition $\E^y(\tau_y)=0$ and Lemma \ref{mvp} we get $T^y_x:=\E^x(\tau_y)$ as the solution of
\begin{equation}\label{LGS}
\left(\begin{array}{c}T^y_1\\ \vdots \\ T^y_n \end{array}\right) = \left(\begin{array}{ccc}p_{11}&\cdots& p_{1n}\\ \vdots&\ddots&\vdots \\ p_{n1} & \cdots & p_{nn} \end{array}\right) \cdot \left(\begin{array}{c}T^y_1\\ \vdots \\ T^y_n\end{array}\right) + \left(\begin{array}{c} 1\\ \vdots \\ 1\end{array}\right).
\end{equation}
Let $T^y:=(T^y_1,\ldots,T^y_n)^T, \ \underline 1:=(1,\ldots,1)^T\in \mathbb{R}^n$, then we can write equation (\ref{LGS}) as
\[T^y=P\cdot T^y +\underline1\ \Leftrightarrow \ (P-I_n)\cdot T^y =-\underline 1,\]
hence the expected hitting times are the solution of an inhomogeneous system of linear equations.

By the linearity of expectations we can determine the commute time between two nodes. Let $\tau_x^y\coloneqq\inf\{l\geq 0 : X_l=x \ \text{and}\ \exists k\leq l \ \text{with} \ X_k=y\}$ be the time the the Markov chain $(X_n)_{n\in \N_0}$ needs to reach node $y$ and then node $x$. The expected commute time $\E^x(\tau_x^y)$ is then given by
\[\E^x(\tau_x^y)=\E^x(\tau_y)+\E^y(\tau_x).\]

\subsection{Graphs as electrical networks}
As a short intermezzo we want to generalize the results of Tetali '91 \cite{Tetali} to weighted graphs. Therefore we want to consider the connected graph $G=(V,E,C)$ as an electrical network with $n$ nodes and $m$ edges. The graph itself is still undirected, even though at some points we will look at directed edges since the current on every edge is only flowing in one direction. When the direction of an edge is important we will consider $(x,y)$ and $(y,x)$ as two different edges. The conductances $(c_{xy})_{x,y\in V}$ are the upper threshold for the current on an edge and $r_{xy}=c^{-1}_{xy}$ is the resistance of the edge $\{x,y\}$. The weight of a node is still the sum over all conductances of incident edges to the node, hence
\[\mu_x=\sum_{y\sim x} c_{xy}.\]
For two adjacent nodes $x$ and $y$ we call $i_{xy}$ the current flowing from $x$ to $y$. Let $x,y\in V$ be two nodes we then call $V_x$ the potential of $x$ and $V_y$ the potential of $y$. The potential between $x$ and $y$ is given by $V_{xy}:=V_x-V_y$ and for $w,z\in V$ it holds
\begin{equation}\label{Diff}V_{wz}=V_w-V_z+V_y-V_y=V_w-V_y-(V_z-V_y)=V_{wy}-V_{zy}.\end{equation}
By Ohm's law, see for example \cite{Hayt}, we have
\begin{equation}\label{Ohm}V_{xy}=r_{xy}i_{xy} \text{ for } x\sim y,\end{equation}
and Kirchhoff's first law states that the current flowing into an inner node of the network is the same as the one flowing out of it, hence if the potential lies on $x$ and $y$ it holds for all $z\in V\setminus\{x,y\}$ that
\begin{equation}\label{Kirch}\sum_{w\sim z} i_{wz}=0.\end{equation}
We now can proof the following lemma.
\begin{lemma}\label{Kirch1}
Let $G=(V,E,C)$ be an electrical network with potentials $V_x>0$ in $x$ and $V_y=0$ in $y$, hence current flowing from $x$ to $y$. Then for all nodes $z\in V\setminus \{x,y\}$
\[V_{zy}=\sum_{w\in V} p_{zw}V_{wy}.\]
\end{lemma}
\begin{proof}
By Kirchhoff's first law (\ref{Kirch}) and Ohm's law (\ref{Ohm}) we get
\begin{eqnarray*}
0&\overset{(\ref{Kirch})}{=}&\sum_{w\sim z}i_{zw}\overset{(\ref{Ohm})}{=}\sum_{w\sim z}\frac{V_{zw}}{r_{zw}}\\
&=&\sum_{w\sim z}V_{zw}c_{zw}\overset{(\ref{Diff})}{=}\sum_{w\sim z}(V_{zy}-V_{wy})c_{zw}\\
&=&\sum_{w\sim z}V_{zy}c_{zw}-\sum_{w\sim z}V_{wy}c_{zw}\\
&=&V_{zy}\mu_z-\sum_{w\sim z}V_{wy}c_{zw}.
\end{eqnarray*}
By solving for $V_{zy}$ we get the desired statement
\[V_{zy}=\sum_{w\sim z}V_{wy}\frac{c_{zw}}{\mu_z}=\sum_{w\in V}p_{zw}V_{wy}.\]
\end{proof}
We now want to draw a connection between random walks on graphs and electrical networks. We therefore define a random walk on $G$ from $x$ to $y$ as the stopped Markov chain with start in $x$ which is stopped when reaching $y$. Hence let $(X_n)_{n\in \N_0}$ be a Markov chain on $G$, $\tau_y$ the hitting time of $y$ and
\[N_z^{xy}:=\sum_{k=0}^{\tau_y-1}\1_{\{X_k=z|X_0=x\}}=\sum_{k=1}^{\tau_y-1}\1_{\{X_k=z|X_0=x\}}+\1_{\{z=x\}}, \ z\in V\setminus \{y\}\]
the number of visits in $z$ of a random walk from $x$ to $y$.

\begin{lemma}
Let $U_z^{xy}:=\E(N_z^{xy})$ then for all $z\in V\setminus\{x,y\}$ we have
\[U_z^{xy}=\sum_{w\in V}U_w^{xy}p_{wz}.\]
\end{lemma}
\begin{proof}
Since $z\not=x,y$ for $N_z^{xy}$ it holds
\begin{eqnarray*}
N_z^{xy}&=&\sum_{k=1}^{\tau_y-1}\1_{\{X_k=z\vert X_0=x\}}+\1_{\{z=y\}}=\sum_{k=1}^{\tau_y-1}\1_{\{X_k=z\vert X_0=x\}}+\1_{\{X_{\tau_y}=z\vert X_0=x\}}\\
&=&\sum_{k=1}^{\tau_y}\1_{\{X_k=z\vert X_0=x\}}.
\end{eqnarray*}
With that and our notation $U_z^{xy}= \E(N_z^{xy})$ we get
\begin{eqnarray*}
U_z^{xy}&=& \E\left(\sum_{k=1}^{\tau_y}\1_{\{X_k=z\vert X_0=x\}}\right)\\
&=&\sum_{n=0}^{\infty}\E\left(\left.\sum_{k=1}^{\tau_y}\1_{\{X_k=z\vert X_0=x\}}\right\vert \tau_y=n\right)\PX(\tau_y=n)\\
 &=&\sum_{n=0}^{\infty}\sum_{k=1}^{n}\E\left(\1_{\{X_k=z\vert X_0=x\}}\right)\PX(\tau_y=n)=\sum_{n=0}^{\infty}\sum_{k=1}^{n}p_{xz}^{(k)}\PX(\tau_y=n)\\
&=&\sum_{n=0}^{\infty}\sum_{k=1}^{n}\sum_{w\in V}p_{xw}^{(k-1)}p_{wz}\PX(\tau_y=n) =\sum_{w\in V}p_{wz}\sum_{n=0}^{\infty}\sum_{k=0}^{n-1}p_{xw}^{(k)}\PX(\tau_y=n)\\
&=&\sum_{w\in V}p_{wz}\sum_{n=0}^{\infty}\sum_{k=0}^{n-1}\E\left(\1_{\{X_{k}=w\vert X_0=x\}}\right)\PX(\tau_y=n)\\
&=&\sum_{w\in V}p_{wz}\E\left( \sum_{k=0}^{\tau_y-1}\1_{\{X_{k}=w\vert X_0=x\}}\right) =\sum_{w\in V}p_{wz}U_w^{xy}.\\
\end{eqnarray*}
\end{proof}
By dividing both sides by $\mu_z$ we get
\[\frac{U_z^{xy}}{\mu_z}=\sum_{w\in V}\frac{U_w^{xy}}{\mu_w}\frac{c_{wz}}{\mu_z}\ \text{ for all } z\in V\setminus\{x,y\}.\]
With the property of lemma \ref{Kirch1} for the potential
\[V_{zy}=\sum_{w\in V}p_{zw}V_{wy}\]
and by choosing $V_{yy}=0,\ V_{xy}=\frac{U^{xy}_z}{\mu_x}$ we get by the uniqueness of harmonic functions that
\[V_{zy}=\frac{U_z^{xy}}{\mu_z}\ \ \forall z\in V.\]
The current of an edge $(w,z)$ equals
\[i_{wz}=\frac{V_{wz}}{r_{wz}}=V_{wy}c_{wz}-V_{zy}c_{wz}=\frac{U_w^{xy}c_{wz}}{\mu_w}-\frac{U_z^{xy}c_{wz}}{\mu_z}=U_w^{xy}p_{wz}-U_z^{xy}p_{zw}\]
and is therefore the expected number of times the random walk traverses the edge $(w,z)$. A random walk starting in $x$ is leaving this node effectively one time and hence the total current leaving $x$ is $1$. Respectively the total current flowing into $y$ is also equal to $1$, which means there's a unit current flowing through the network:
\[\sum_{w\sim x}i_{xw}=1= \sum_{z\sim y}i_{zy}.\]
This yields by Ohm's law (\ref{Ohm}) that the effective resistance between $x$ and $y$ -- denoted by $R_{xy}$ -- is exactly the potential between these two nodes $V_{xy}$.
Let $U_w=U^{xy}_w+U^{yx}_w$ be the number of times a random walk starting from $x$ going to $y$ and returning to $x$ is visiting $w$. We then get
\[U_w=U_w^{xy}+U_w^{yx}=V_{wy}\mu_w-V_{wx}\mu_w=(V_{wy}+V_{xw})\mu_w=V_{xy}\mu_w=R_{xy}\mu_w,\]
where $U_w^{yx}=-V_{wx}\mu_w$ since the random walk is walking in the opposite direction in which the current flows.
We also can write the expected hitting time $\E^x(\tau_y)$ as the expected number of times a random walk from $x$ to $y$ visits every node in the network, hence
\[\E^x(\tau_y)=\sum_{w \in V}U^{xy}_w.\]
Then for the commute time between $x$ and $y$ it holds
\begin{eqnarray}
\label{RT}\E^x(\tau_x^y)&=&\E^x(\tau_y)+\E^y(\tau_x)=\sum_{w\in V}U^{xy}_w+\sum_{w\in V}U_w^{yx}\nonumber \\
&=&\sum_{w\in V}U_w=R_{xy}\sum_{w\in V}\mu_w.
\end{eqnarray}

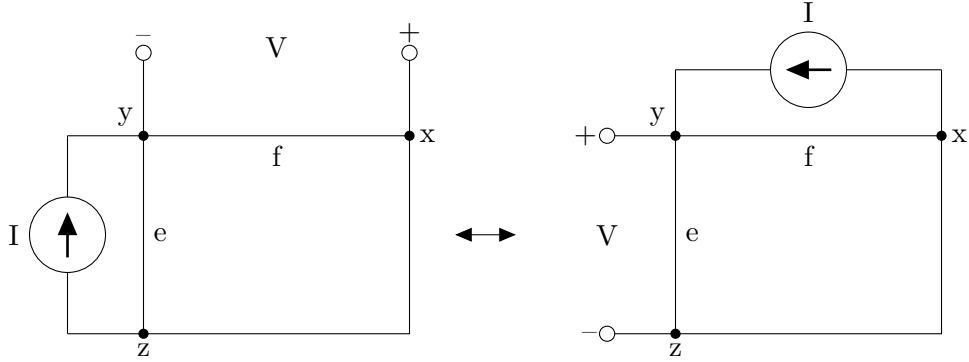
\begin{figure}[h]
\centering
\begin{tikzpicture}
\draw (0,0) rectangle +(3.5,2.625);
\draw (0,0)--(-1,0);
\draw (0,2.625)--(-1,2.625);
\draw (0,2.625)--(0,3.625);
\draw (3.5,2.625)--(3.5,3.625);
\draw (0,3.725) circle (1mm) node[above]{--};
\draw (3.5,3.725) circle (1mm)node[above]{+};
\draw (-1,0)--(-1,2.625);
\filldraw[fill=white] (-1,1.3125) circle (5mm);
\draw (-1.5,1.3125) node[left]{I};
\draw[triangle 45-,line width=1pt] (-1,1.6125)--(-1,1.0125);
\draw (1.75,3.8) node{V};
\draw (0,1.3125) node[right]{e};
\draw (1.75,2.625) node[below]{f};
\filldraw[fill=black] (3.5,2.625) circle(0.6mm) node[right]{x};
\filldraw[fill=black] (0,0) circle(0.6mm) node[below]{z};
\filldraw[fill=black] (0,2.625) circle(0.6mm) node[anchor=south east]{y};

\draw[triangle 45-triangle 45] (4.1,1.3125)--(4.9,1.3125);

\draw (7,0) rectangle +(3.5,2.625);
\draw (7,0)--(6.2,0);
\draw (7,2.625)--(6.2,2.625);
\draw (7,2.625)--(7,3.5);
\draw (10.5,2.625)--(10.5,3.5);
\draw (6.1,0) circle (1mm) node[left]{--};
\draw (6.1,2.625) circle (1mm) node[left]{+};
\draw (7,3.5)--(10.5,3.5);
\filldraw[fill=white] (8.75,3.5) circle (5mm);
\draw (8.75,4) node[above]{I};
\draw[triangle 45-,line width=1pt] (8.45,3.5)--(9.05,3.5);
\draw (6.1,1.3125) node{V};
\draw (7,1.3125) node[right]{e};
\draw (8.75,2.625) node[below]{f};
\filldraw[fill=black] (10.5,2.625) circle(0.6mm) node[right]{x};
\filldraw[fill=black] (7,0) circle(0.6mm) node[below]{z};
\filldraw[fill=black] (7,2.625) circle(0.6mm) node[anchor=south east]{y};
\end{tikzpicture}
\caption{Reciprocity in electrical networks (see \cite{Tetali}).}\label{reciprocity}
\end{figure}

The reciprocity in electrical networks gives us that the potential $V_{zy}$ for a current flowing between $x$ and $y$ is the same as the potential $V_{xy}$ if the current flows between $z$ and $y$ (see figure \ref{reciprocity}).
For random walks we get that the number of visits in $z$ proportional to its weight when walking from $x$ to $y$ is the same as the number of visits in $x$ proportional to its weight when walking from $z$ to $y$, hence
\[V_{zy}=\frac{U^{xy}_z}{\mu_z}=\frac{U^{zy}_x}{\mu_x}=V_{xy}.\]
With that we are able to proof our final theorem of this section.

\begin{thm}\label{commute}
Let $G=(V,E,C)$ be a finite graph with $n\in \N$ nodes. It then holds
\[\sum_{(x,y)\in E}\E^x(\tau_y)\frac{c_{xy}}{\sum_{(w,u)\in E}c_{wu}}=n-1.\]
\end{thm}
\begin{proof}
By using the reciprocity we get
\[\sum_{x\sim y}\frac{U_z^{xy}c_{xy}}{\mu_z}=\sum_{x\sim y}\frac{U_x^{zy}c_{xy}}{\mu_x}=\sum_{x\sim y}U_x^{zy}p_{xy}=\left\{\begin{array}{ll}1& \text{ for } z\not=y\\ 0& \text{ for } z=y\end{array}\right.,\] 
since the expected number of times $y$ is reached from one of its neighbors is exactly $1$ if the random walk does not start in $y$. Summing over all possible terminal nodes $y$ yields
\[\sum_{y\in V}\sum_{x\sim y}\frac{U_z^{xy}c_{xy}}{\mu_z}=\sum_{y \in V}\1_{\{y\not=z\}}=n-1,\ z\in V.\]
We can simplify this by considering two random walks, one going from $x$ to $y$ and the other on going from $y$ to $x$. This gives us
\begin{eqnarray*}
n-1&=&\sum_{y\in V}\sum_{x\sim y}\frac{U_z^{xy}c_{xy}}{\mu_z}=\sum_{(x,y)\in E}\frac{U_z^{xy}c_{xy}}{\mu_z}=\sum_{\{x,y\}\in E}\left(\frac{U_z^{xy}c_{xy}}{\mu_z}+\frac{U_z^{yx}c_{yx}}{\mu_z}\right)\\
&=&\sum_{\{x,y\}\in E}\frac{U_z}{\mu_z}c_{xy}=\sum_{\{x,y\}\in E}\frac{R_{xy}\mu_z}{\mu_z}c_{xy}=\sum_{\{x,y\}\in E}R_{xy}c_{xy}.
\end{eqnarray*}
From equation (\ref{RT}) we know
\[\E^x(\tau_x^y)=R_{xy}\sum_{w\in V}\mu_w.\]
By multiplication with $c_{xy}$ and summing over all edges we get
\[\sum_{\{x,y\}\in E}\E^x(\tau_x^y)c_{xy}=\left(\sum_{w\in V}\mu_w\right)\cdot\left(\sum_{\{x,y\}\in E}R_{xy}c_{xy}\right)=(n-1)\sum_{w\in V}\mu_w,\]
which yields
\[\frac{\sum\limits_{\{x,y\}\in E}\E^x(\tau_x^y)c_{xy}}{\sum\limits_{w\in V}\mu_w}=n-1.\]
If we finally consider directed edges the desired equation follows:
\begin{eqnarray*}
n-1&=&\sum\limits_{\{x,y\}\in E}\E^x(\tau_x^y)\frac{c_{xy}}{\sum\limits_{w\in V}\sum\limits_{u\sim w}c_{wu}}\\
&=&\sum\limits_{\{x,y\}\in E}\left(\E^x(\tau_y)\frac{c_{xy}}{\sum\limits_{(w,u)\in E}c_{wu}}+\E^y(\tau_x)\frac{c_{yx}}{\sum\limits_{(w,u)\in E}c_{wu}}\right)\\
&=&\sum\limits_{(x,y)\in E}\E^x(\tau_y)\frac{c_{xy}}{\sum\limits_{(w,u)\in E}c_{wu}}.
\end{eqnarray*}
\end{proof}
For unweighted graphs the statement of Theorem \ref{commute} simplifies to
\[\sum\limits_{(x,y)\in E}\E^x(\tau_y)\frac{1}{2m}=n-1,\]
where $m$ is the number of edges in the graph.

\section{Models}\label{models}
There are different possibilities for modeling networks. We consider two models in order to analyze the resulting graphs. Firstly the \ER\ model in which every two nodes are independently of each other connected with the same probability and secondly the \PA \ model where the probability of two nodes being connected depends on the current degrees of the nodes.\\
We also take a look at subgraphs of those random graphs in order to analyze if and how certain properties are inherited from the original graph. This is useful when considering networks where not the whole network is known like in the case of the protein network in Section 4.

\subsection{The \ER\ graph}

\begin{figure}[t]
\centering
\includegraphics[width=4cm]{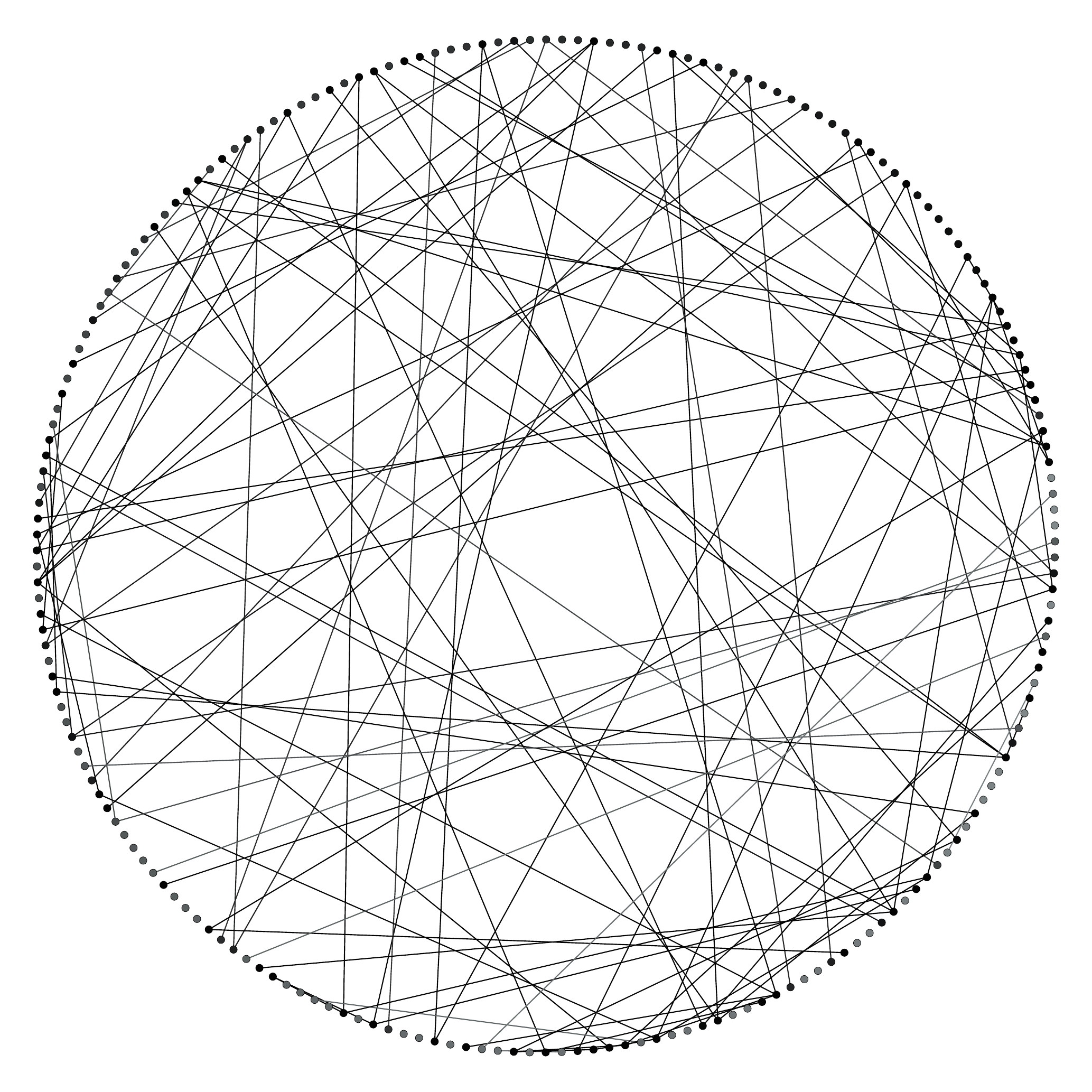}
\includegraphics[width=4cm]{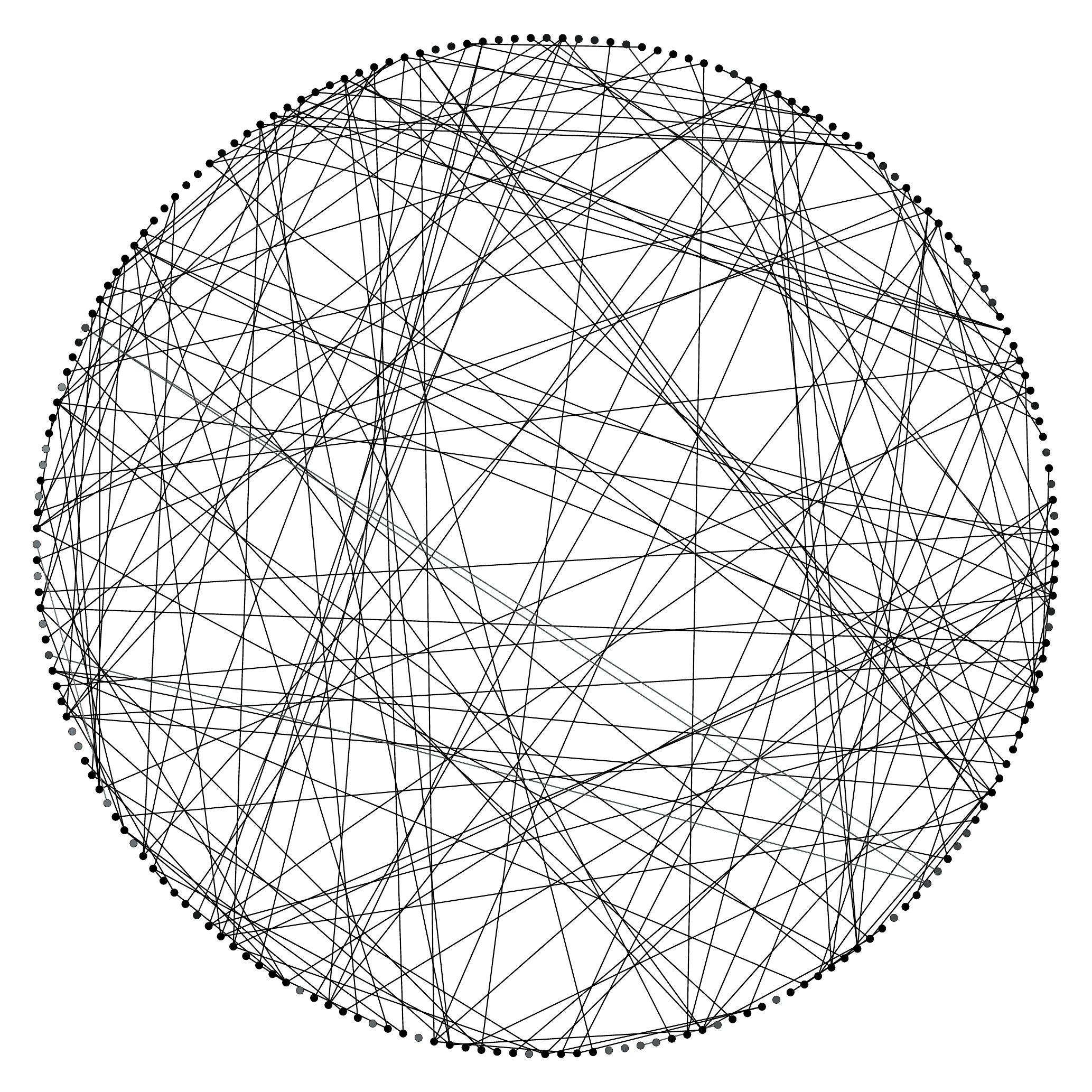}
\includegraphics[width=4cm]{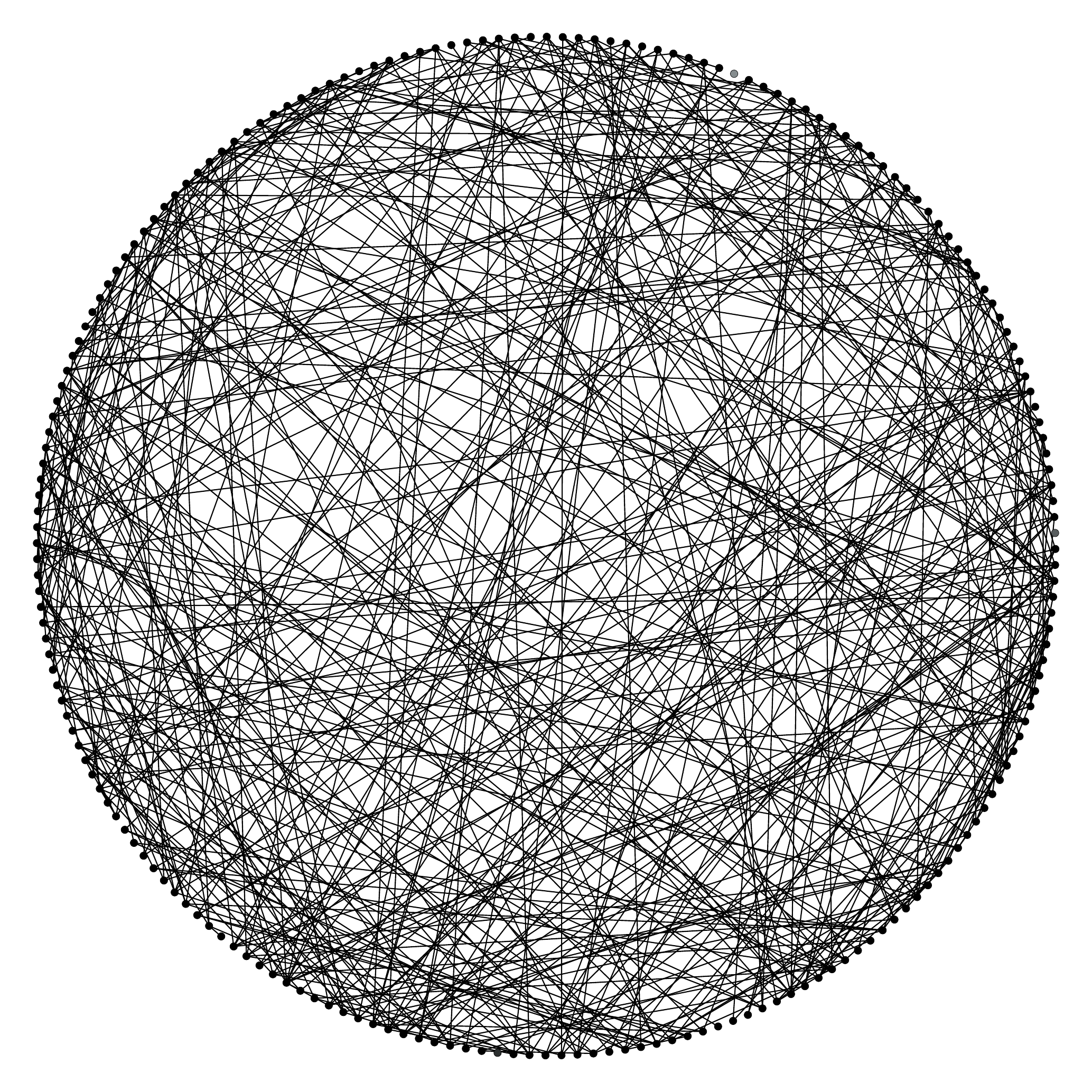}
\caption{Three \ER \ graphs with $n=200$ nodes and edge probabilities $p=\frac{1}{200},\frac{2}{200},\frac{5}{200}$ (left to right).}\label{ERNW}
\end{figure}

An \ER \ graph is a graph $G=(V,E)$ with $V=[n],n \in \N$. Let $(Y_{ij})_{1\leq i< j\leq n}$ be i.i.d. random variables with $Y_{12}\sim Bin(1,p), p\in [0,1]$. The edge set is then given by $E=\{\{i,j\}\in V\times V :  Y_{ij}=1\}.$ Let 
\[X_{ij}\coloneqq \left\{\begin{array}{ll} Y_{ij} & \text{if }i<j\\ 0 & \text{if }i=j \\ Y_{ji} & \text{if }i>j,\end{array}\right. \]
then the adjacency matrix $A$ of $G$ is given by $A=(X_{ij})_{i,j\in V}$. In figure \ref{ERNW} are three different \ER\ graphs.
By above construction it is obvious that the degree distribution of the \ER \ graph is again a binomial distribution, hence $D(i)\sim Bin(n-1,p)$ for all $i\in [n]$. Further we have the mean degree in an \ER \ graph given as $\E (D(i))=(n-1)p$.

\begin{figure}[p]
\centering
\includegraphics[width=6cm]{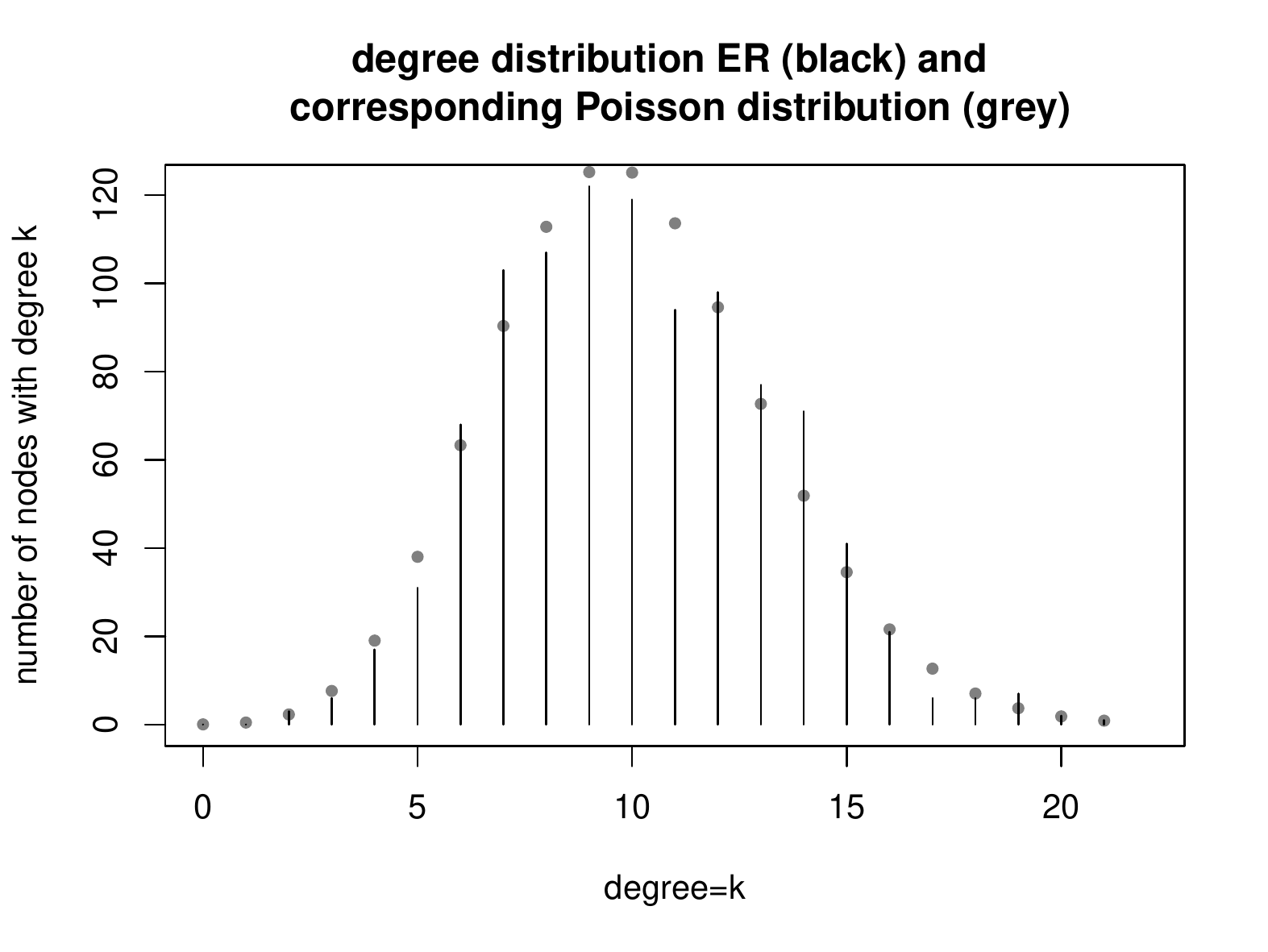}
\includegraphics[width=6cm]{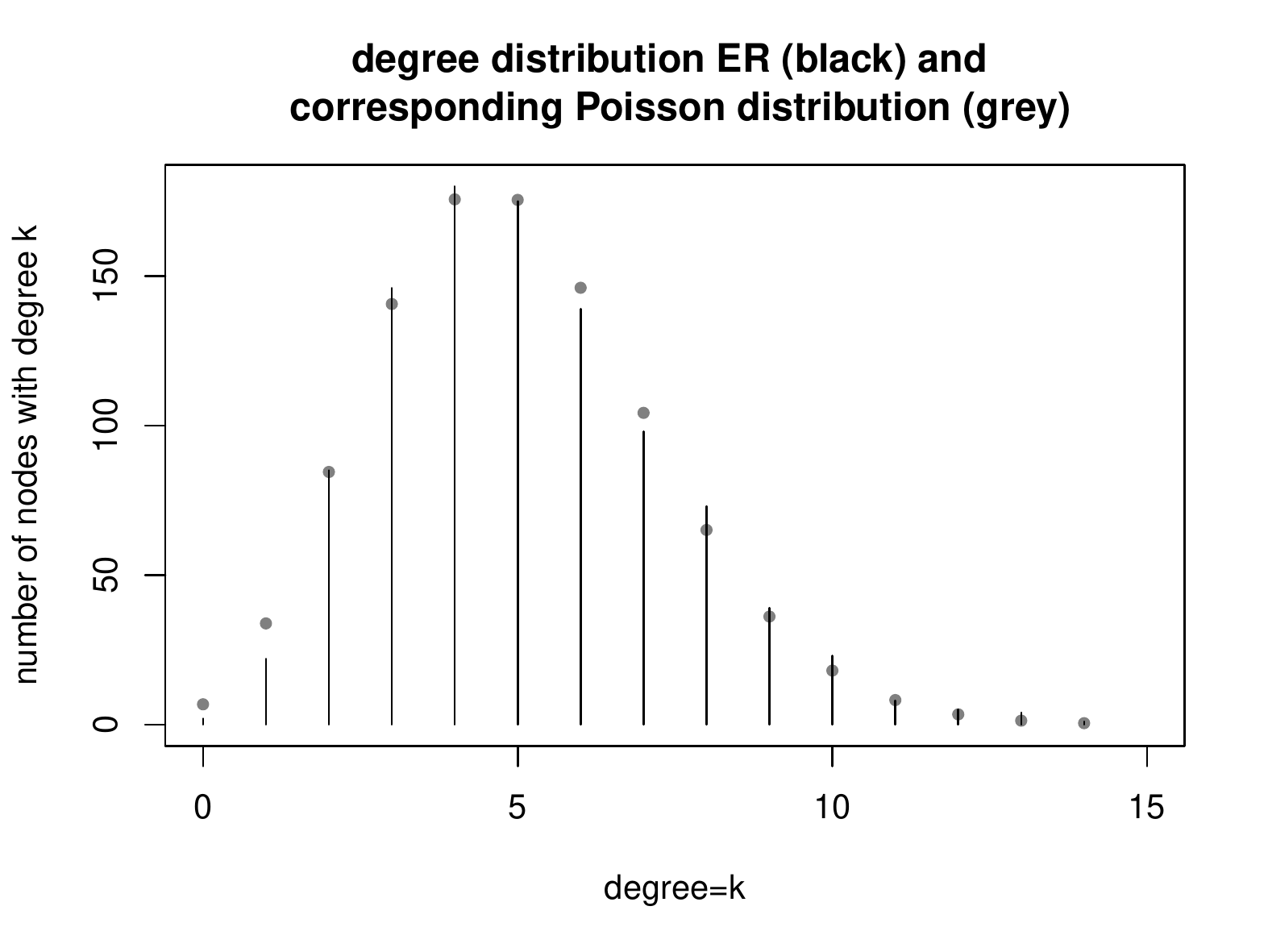}
\caption{On the left is the degree distribution of an \ER \ graph with $n=1000$ nodes and $p=\frac{1}{100}$. The dots represent the corresponding Poisson distribution with $\lambda=(n-1)p=9.99$. All \ER \ subgraphs are based on this \ER \ graph.\newline
On the right is the degree distribution of a subgraph constructed by selection of edges with probability $q=0.5$ again with the corresponding Poisson distribution with $\lambda=(n-1)pq=4.995$.}\label{ERdeg}
\vspace{1cm}
\includegraphics[width=6cm]{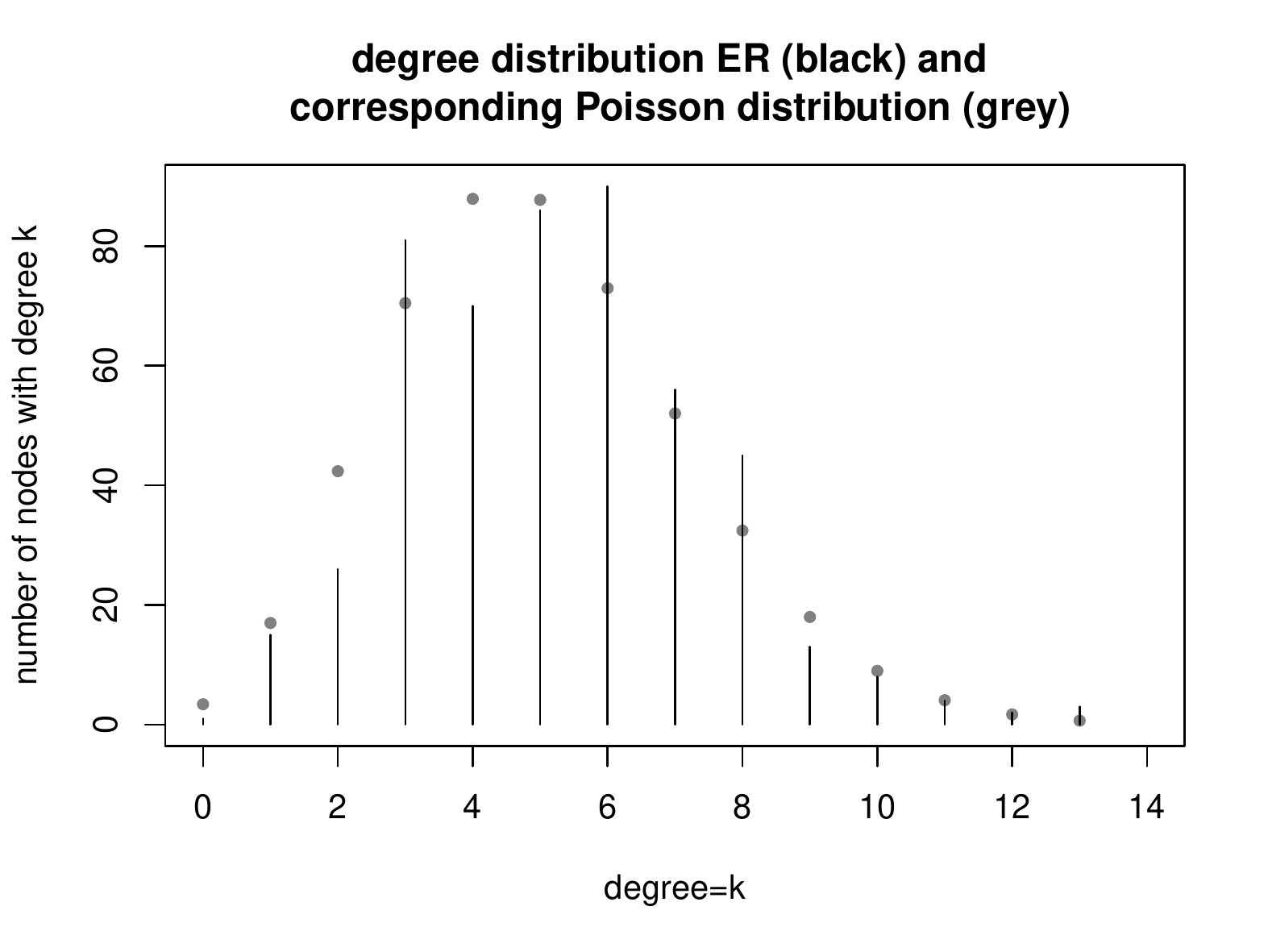}
\includegraphics[width=6cm]{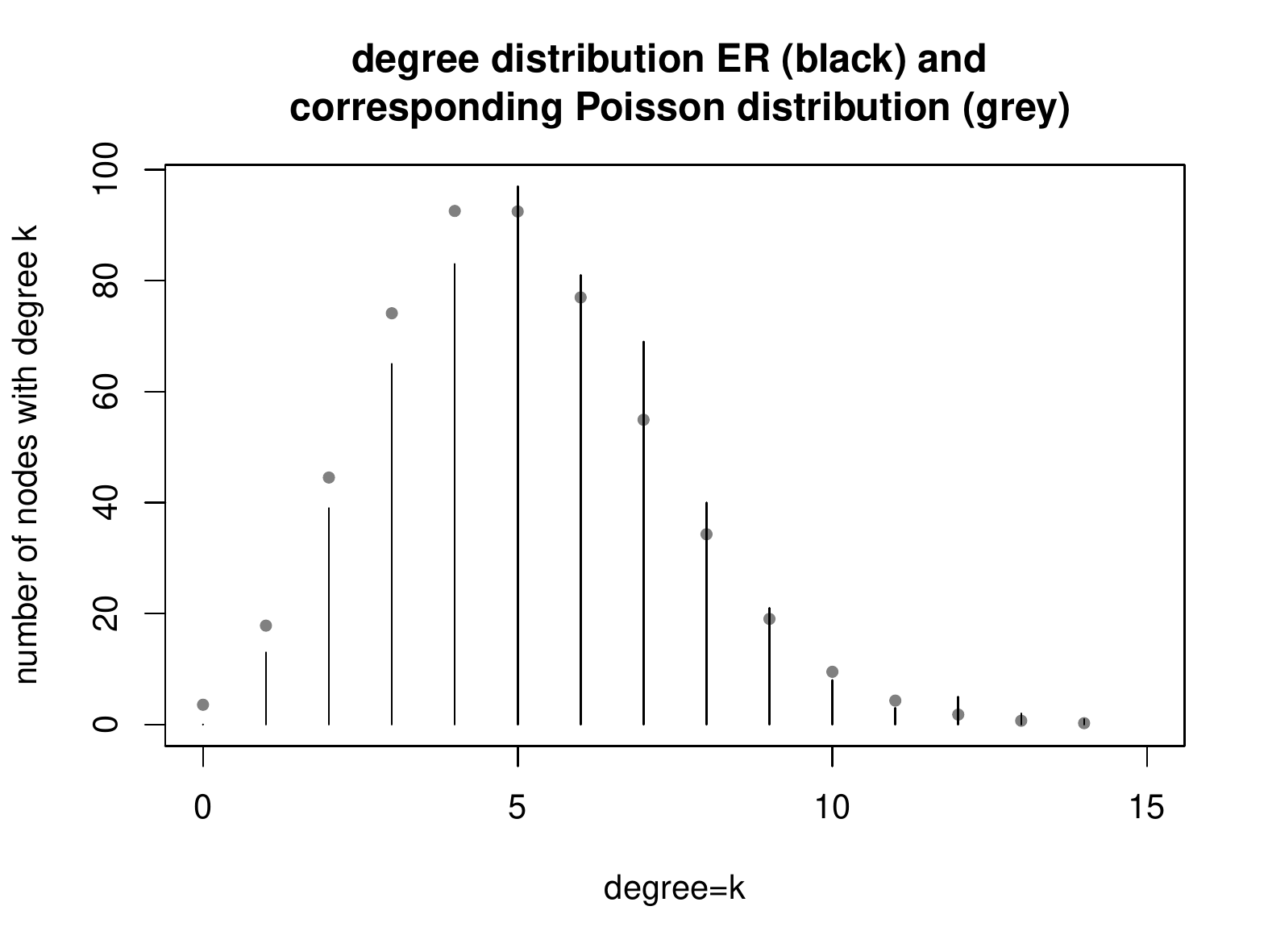}
\caption{On the left we see the degree distribution of a subgraph constructed by uniform selection of nodes with $m=500$ together with the corresponding Poisson distribution with $\lambda=(m-1)p=4.99$.\newline
On the right is the degree distribution of a subgraph constructed by binomial selection of nodes with probability $q=\frac{m}{n}=0.5$ and the corresponding Poisson distribution with $\lambda=(n-1)pq=4.995$.} \label{degER}
\end{figure}

\subsubsection{Subgraphs of \ER \ graphs}
We now want to analyze subgraphs of \ER\ graph and especially their degree distribution. We consider three different mechanisms for the constructions of our subgraphs. Let first $q\in [0,1]$ be the probability for an edge from the graph $G$ to be in the subgraph. We then decide for every edge of $G$ independently if it should stay in the subgraph, i.e.
\[\PX(D_{Sub}(i)=l\vert D(i)=k)={k \choose l} q^l(1-q)^{k-l},\ k\in\{0,\ldots,n-1\},\ l\in\{0,\ldots,k\}.\]
where $D_{Sub}(i)$ denotes the degree of node $i\in [n]$ in the subgraph. Then by the law of total probability it follows
\begin{eqnarray*}
&&\PX(D_{Sub}(i)=l)\\
&=&\sum_{k=0}^{n-1}\PX(D_{Sub}(i)=l\vert D(i)=k)\PX(D(i)=k)\\
&=&\sum_{k=l}^{n-1}{k \choose l}q^l(1-q)^{k-l}{n-1 \choose k}p^k(1-p)^{n-1-k}\\
&=&{n-1 \choose l} (pq)^l \sum_{k=0}^{n-1-l} { n-1-l \choose k}(p-pq)^{k}(1-p)^{n-1-l-k}\\
&=&{n-1 \choose l} (pq)^l (1-pq)^{n-1-l}.\\
\end{eqnarray*}
Hence the degree distribution of the subgraph is again binomial with parameters $n-1$ and $pq$. The degree distributions of an \ER\ graph and the resulting subgraph are depicted in figure \ref{degER}.

Next we construct subgraphs by deleting nodes from the graph . In this case we keep an edge if both adjacent nodes are also in the subgraph. One possibility to do this is to fix the number of nodes in the subgraph and then choosing the subgraph from the set of all subgraph with this number of nodes.\\
Hence let $G=(V,E)$ be an \ER \ graph with $V=[n]$. Let further $m\in[n]$ be the number of nodes in the subgraph and
\[\Omega_m\coloneqq \left\{\omega=(\omega_1,\ldots,\omega_n)\in \{0,1\}^n: \sum_{i=1}^n \omega_i =m\right\}, \]
the set of all possibilities of choosing $m$ nodes out of $n$ nodes. Since every subgraph with $m$ nodes has the same probability to be chosen it holds for all $\omega\in\Omega_m$
\[\PX(\{\omega\})= \frac{1}{{n\choose m}}.\]
We denote such a subgraph of the \ER \ graph by $G^m_{Sub}=(V^m_{Sub},E^m_{Sub})$. Let $i \in V^m_{Sub}$ be a fixed node then the subset of $\Omega_m$ giving all combinations of the other $m-1$ nodes ist given by
\begin{align*}\Omega_m^i&\coloneqq \{\omega\in\Omega_m: \omega_i=1\}\\
&= \left\{\omega=(\omega_1,\ldots,\omega_n)\in \{0,1\}^{i-1}\times\{1\}\times\{0,1\}^{n-i}: \sum_{j\in[n]\setminus{i}} \omega_j =m-1\right\}.
\end{align*}
Since the number of elements in $\Omega^i_m$ is given by ${n-1 \choose m-1}$ the conditional degree distribution of $i$ is given by
\[\PX(D_{Sub}(i)=l\vert D(i)=k)=\frac{{n-1-k \choose m-1-l}{k \choose l}}{{n-1\choose m-1}}.\]
By using the law of total probability again it holds for fixed $i \in V^m_{Sub}$
\begin{eqnarray*}
&&\PX(D_{Sub}(i)=l)\\
&=& \sum_{k=0}^{n-1}\PX(D_{Sub}(i)=l\vert D(i)=k)\PX(D(i)=k)\\
&=&\sum_{k=0}^{n-1}\frac{{n-1-k \choose m-1-l}{k \choose l}}{{n-1 \choose m-1}}{n-1 \choose k}p^k (1-p)^{n-1-k}\\
&=&\sum_{k=0}^{n-1}{n-m \choose k-l}{m-1\choose l} p^{k-l+l}(1-p)^{n-1-(m-1)-(k-l)+m-1-l}\\
&=&{m-1\choose l}p^l(1-p)^{m-1-l}\sum_{k=0}^{n-m}{n-m \choose k} p^{k}(1-p)^{n-m-k}\\
&=&{m-1\choose l}p^l(1-p)^{m-1-l}.\\
\end{eqnarray*}
Hence we again have binomially distributed degrees with parameters $m-1$ and $p$.\\
We now want to choose nodes binomially distributed to stay in the subgraph. Let therefore be $q\in[0,1]$. We denote the resulting subgraph by $G^q_{Sub}=(V^q_{Sub},E^q_{Sub})$. It then holds for $m\in \{0,\ldots,n\}$:
\[\PX(\# V_{Sub}^q=m)={n\choose m}q^m(1-q)^{n-m}.\]
Let again be $i \in V^q_{Sub}$ be a fixed node then it holds
\[\PX(\# V_{Sub}^q=m)={n-1\choose m-1}q^{m-1}(1-q)^{n-m}, \ m\in\{1,\ldots,n\}\]
and the probability for node $i$ in the subgraph to have degree $l$, given $D(i)=k$ in the graph and $\#V^q_{Sub}=m$ is
\[\PX(D_{Sub}(i)=l\vert D(i)=k,\# V_{Sub}^q=m)=\left\{\begin{array}{ll} \frac{{n-1-k \choose m-1-l}{k \choose l}}{{n-1\choose m-1}} &\text{for } i\in V_{Sub}^q,\\ 0&\text{for } i\notin V_{Sub}^q. \end{array}\right.\]
By the law of total probability we get
\begin{eqnarray*}
&&\PX(D_{Sub}(i)=l\vert \#V_{Sub}^q=m)\\
&=& \sum_{k=0}^{n-1}\PX(D_{Sub}(i)=l\vert D(i)=k,\#V_{Sub}^q=m)\PX(D(i)=k)\\
&=&\sum_{k=0}^{n-1}\frac{{n-1-k \choose m-1-l}{k \choose l}}{{n-1 \choose m-1}}{n-1 \choose k}p^k (1-p)^{n-1-k}={m-1\choose l}p^l(1-p)^{m-1-l},\\
\end{eqnarray*}
which yields
\begin{eqnarray*}
&&\PX(D_{Sub}(i)=l)\\
&=&\sum_{m=l+1}^n\PX(D_{Sub}(i)=l  \vert \#V_{Sub}^q=m)\PX(\#V_{Sub}^q=m)\\
&=&\sum_{m=l+1}^{n} {m-1 \choose l} p^l(1-p)^{m-1-l}{n-1\choose m-1}q^{m-1}(1-q)^{n-m}\\
&=&{n-1 \choose l } (pq)^l\sum_{m=0}^{n-l-1}{n-1-l\choose m}(q-pq)^{m}(1-q)^{n-1-l-m}\\
&=&{n-1 \choose l } (pq)^l(1-pq)^{n-l-1}.\\
\end{eqnarray*}
Hence the degrees in $G^q_{Sub}$ are again binomially distributed with parameters $n-1$ and $pq$. If we now choose $q\coloneqq \frac{m}{n}$ then $G^m_{Sub}$ and $G^q_{Sub}$ are comparable by their number of nodes, because
\[ \E(\#V^q_{Sub})=nq=n\frac{m}{n}=m=\#V^m_{Sub}=\E(\#V^m_{Sub}).\]

\begin{figure}[h]
\centering
\includegraphics[width=8cm]{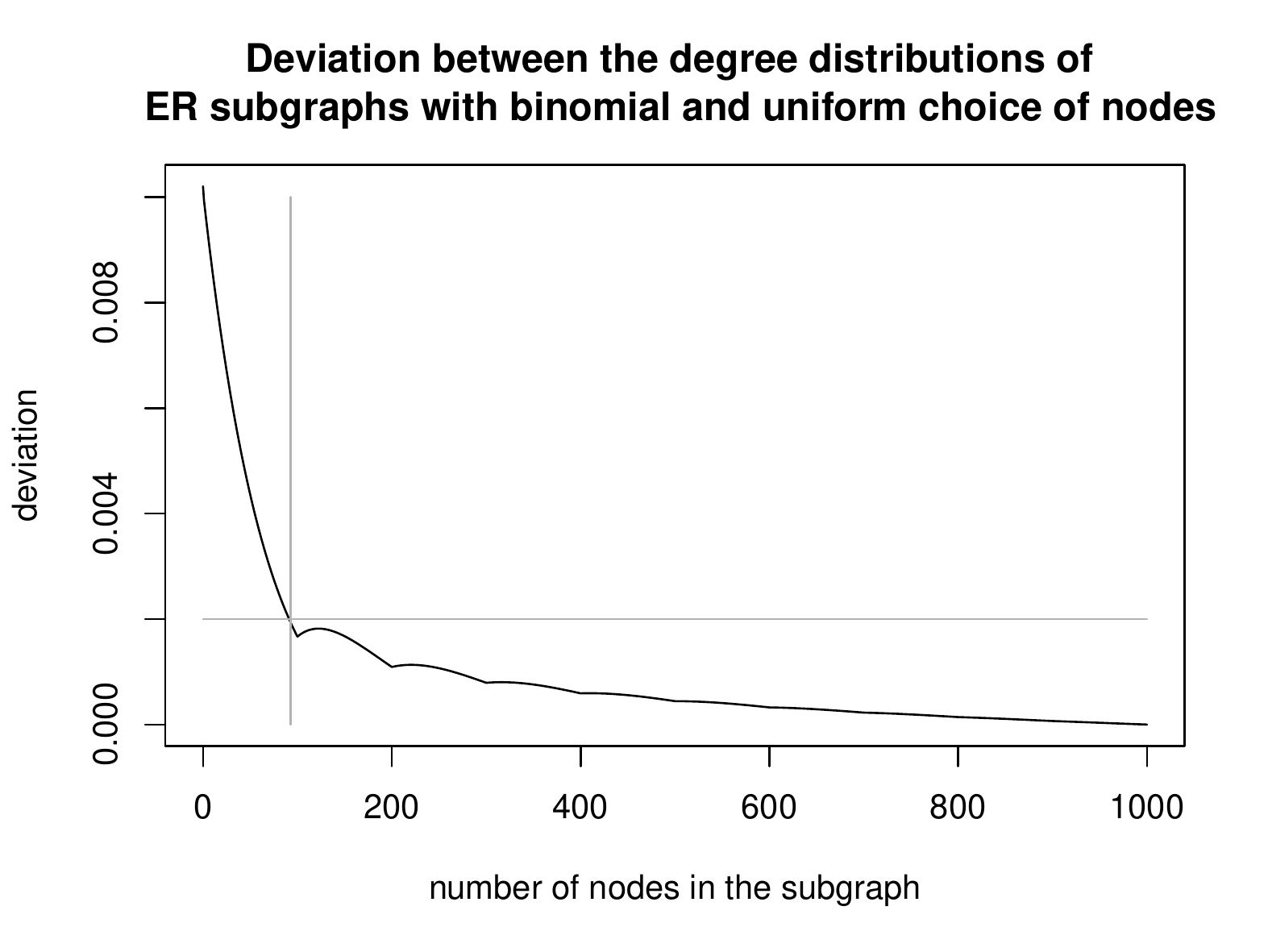}
\caption{Absolute deviation of the degree distributions by constructing the \ER \ subgraphs via the selection of nodes (black line). For subgraphs wit more than $93$ nodes the absolute deviation of the two binomial distributions is smaller than $0.002.$}\label{Fehlerplot}
\end{figure}

Moreover the degree distributions are comparable as figure \ref{Fehlerplot} shows. For an \ER \ graph with $n=1000$ nodes and edge probability $p=1/100$ we get, that the degree distributions of the subgraphs only differ significantly if the subgraphs have less than $10\%$ of the nodes of the graph.\\
By the calculations of this section we see that subgraphs of an \ER \ graph constructed according to one of the above mechanisms have the same structure as the graph itself and can therefore again be modeled as \ER\ graphs.

\subsection{The \PA \ model}
There are different possibilities to define a \PA \ model but the basic idea stays the same. All \PA \ models consider a growing graph where the degrees of the existing nodes influence the probability of a new node connecting to them. In our model we will not allow self-loops, hence every new node really is connected to the existing graph.\\
For our model we fix $m\in \N$ and $\delta\in (-m,\infty)$. We then initialize our graph with one node with $m$ self-loops. Here the self-loops are necessary to calculate the probabilities. Every new node has $m$ edges which connect to existing nodes. At time $n\geq 2$ the $n$-th node is added to the graph.\\
\begin{figure}[hb!]
\centering
\includegraphics[width=5cm]{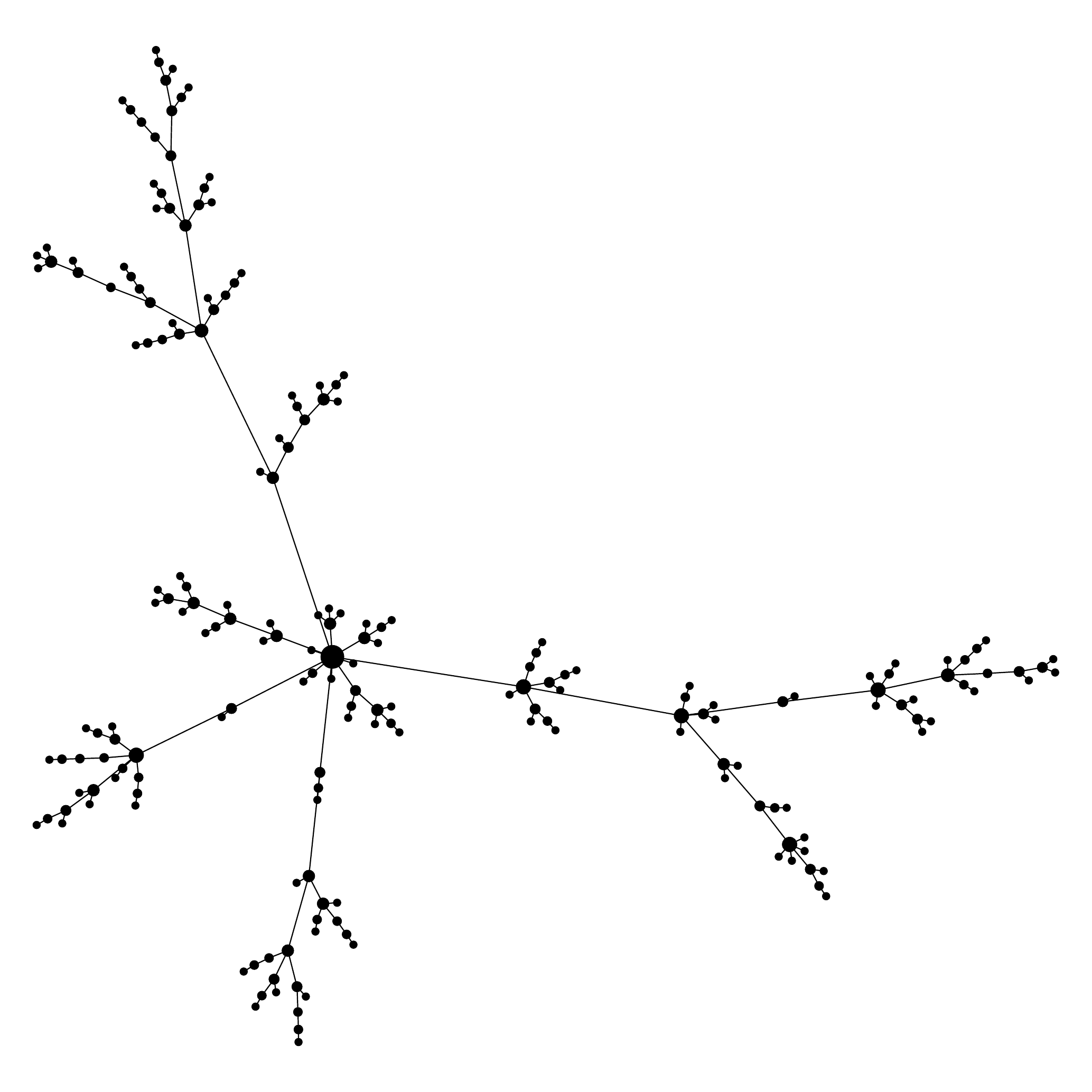}
\includegraphics[width=5cm]{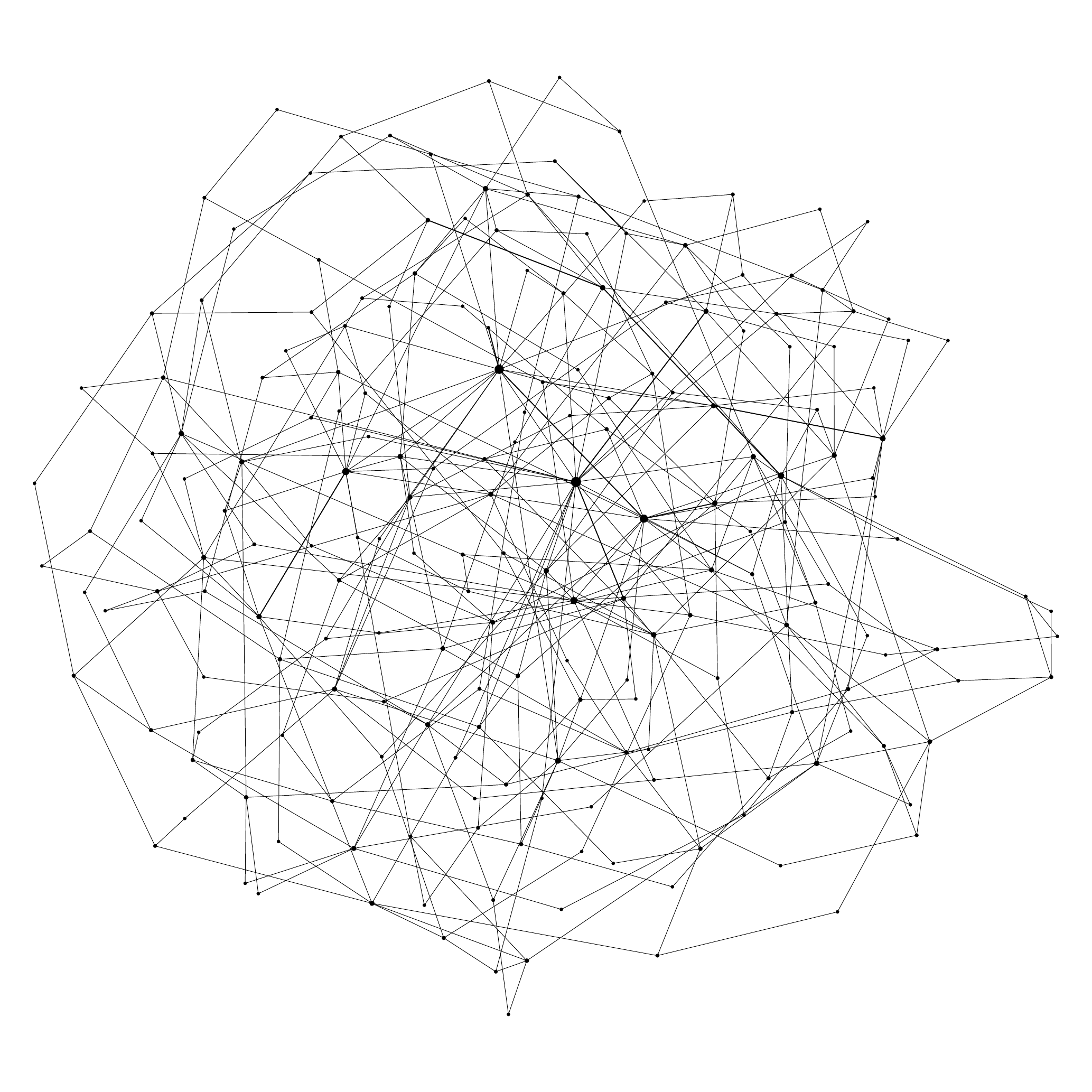}
\includegraphics[width=5cm]{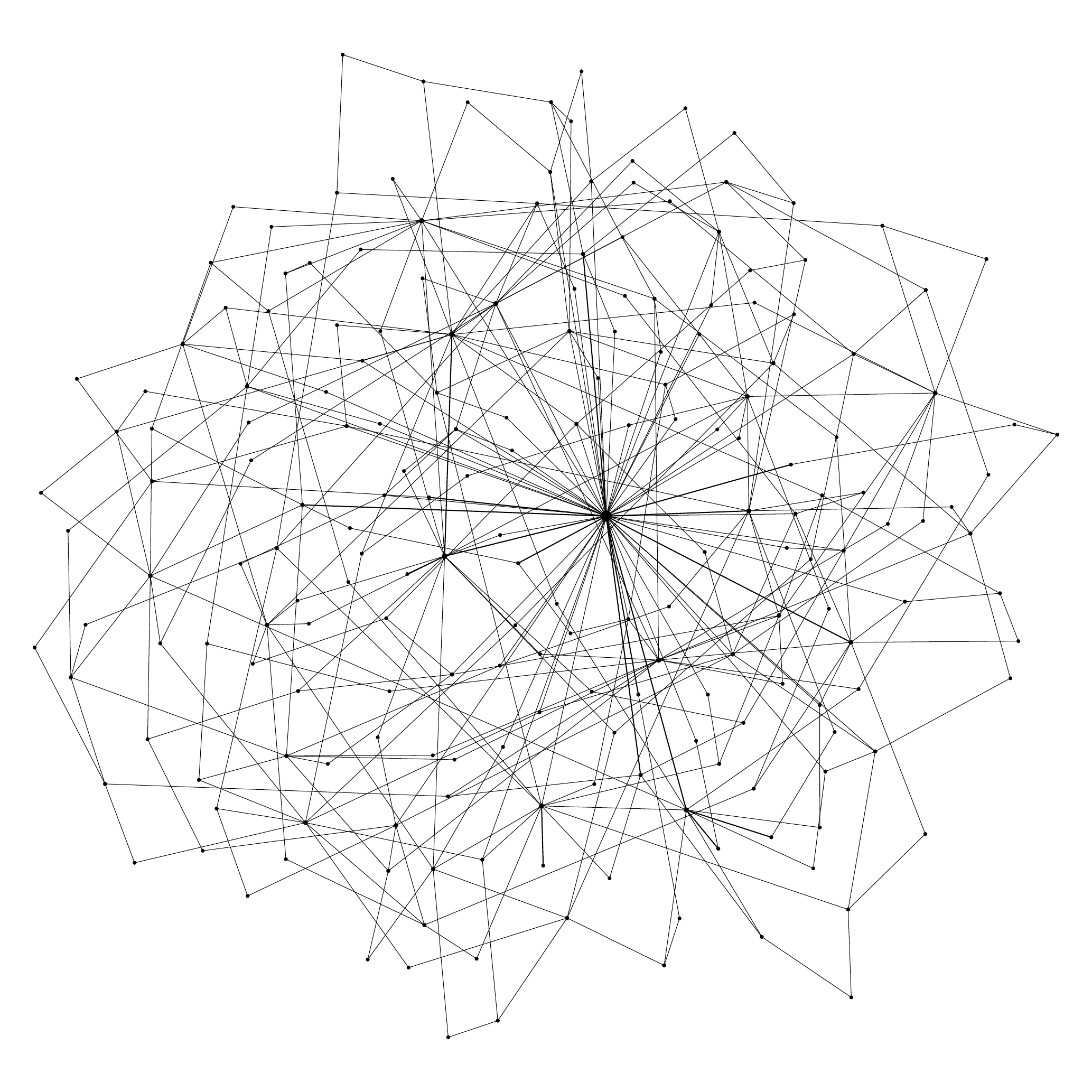}
\includegraphics[width=5cm]{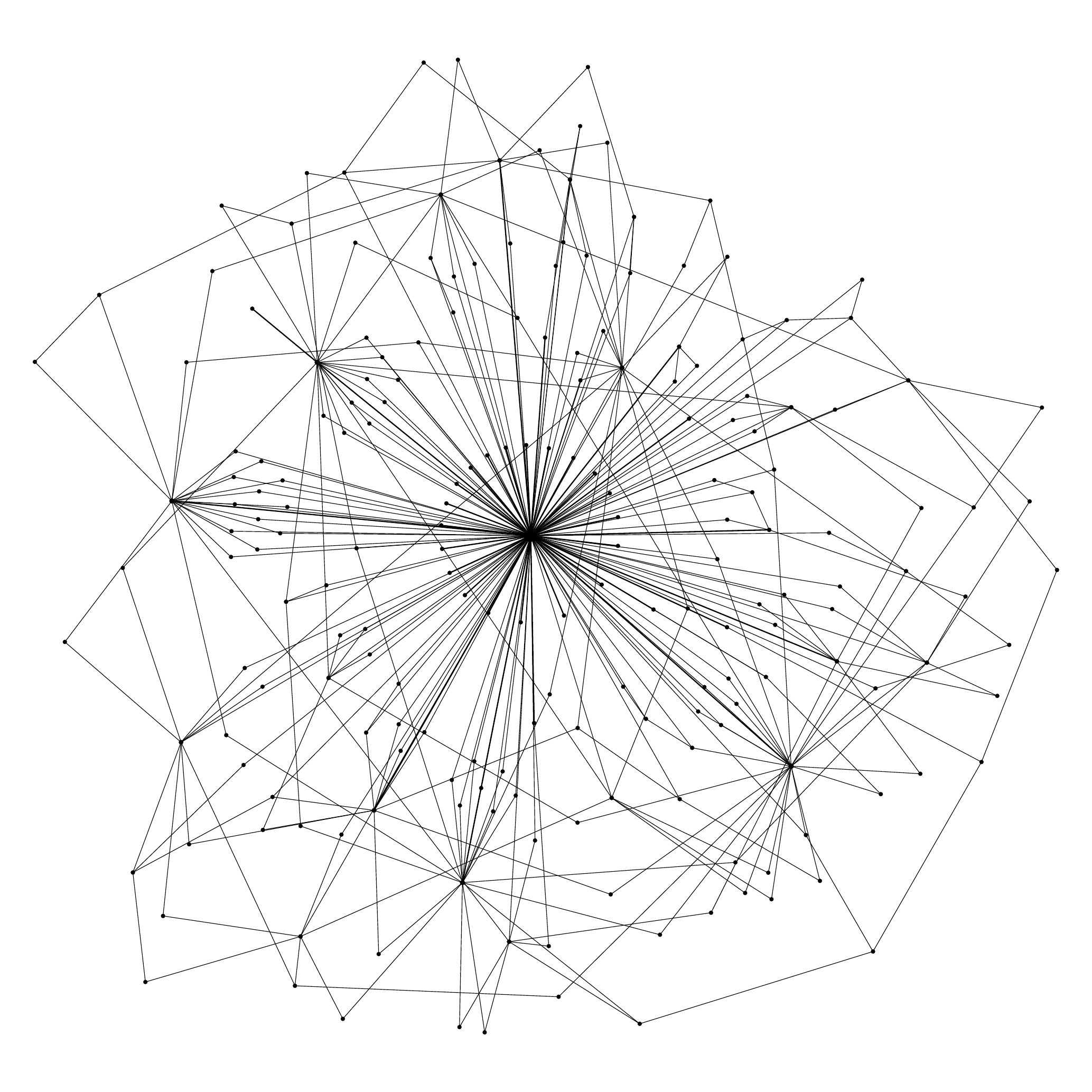}
\caption{Four \PA \ graphs with $n=200$ nodes and different values for $m$ and $\delta$. Upper left: $m=1,\ \delta=10$. Upper right: $m=2,\ \delta=10$. Lower left: $m=2,\ \delta=0$. Lower right: $m=2,\ \delta=-1.5$.}\label{PANW}
\end{figure}

Let $D^n(i)$ be the degree of node $i\leq n$ at time $n$ and $PA_n^{(m,\delta)}=(V_n,E_n)$ the \PA \ graph with parameters $m\geq 1$ and $\delta>-m$ at time $n$. Then the probability of node $n+1$ to connect to node $i$, hence the probability for $\{n+1,i\}\in E_n$ given $PA_n^{(m,\delta)}$, is given by

\[\PX(\{n+1,i\}\in E_{n+1}\vert PA^{(m,\delta)}_n)\coloneqq \frac{D^{n}(i)+\delta}{(2m+\delta)n}\ \text{for} \ i\in [n],\ n\in\N.\]
Since we only consider finite graphs, we stop the construction when the graph has the desired size. By construction the graph is connected, due to the fact that we do not allow self-loops. In figure \ref{PANW} we see four \PA \ graphs with different parameters. This leads to different structural properties. Since $m$ is responsible for the number of edges added with each node one will always get a tree for $m=1$. The parameter $\delta$ controls the influence of the degrees on the connection probabilities. For $\delta$ close to $-m$ the probability for a new node to connect to a node with degree $m$ is rather small which leads to a graph where early nodes are preferred and get a much higher degree than nodes added later on. For $\delta$ large the influence of the degrees on the connection probabilities is small which leads to a more homogeneous graph. 

\subsubsection{Properties of \PA \ graphs}
We now want to look at the degree distribution and because we can not determine it exactly we are also interested in the expected degree and variance of the degree of fixed nodes. First we get that the mean degree in a \PA \ graph of size $n$ is given by
\begin{equation}\label{mean}
\frac{1}{n}\sum_{i=1}^n D^n(i)=\frac{2nm}{n}=2m,
\end{equation}
since every node adds $m$ edges to the graph and the sum over all degrees is twice the number of edges. We are not able determine the degree distribution exactly, but it is possible to show that the degree distributions converge and to give the exact limit.\\
Let $P_k(n):=\frac{1}{n}\sum_{i=1}^n \1_{\{D^n(i)=k\}},\ k\in \N_0$ be the ratio of nodes with degree $k$ at time $n$. Then $(P_k(n))_{k\geq 0}$ defines the degree distribution of $PA_n^{(m,\delta)}$. For $m\geq 1, \ \delta>-m$ and all $k\in \N$ we define the sequence $(p_k)_{k\in \N}$ by
\begin{equation}
\label{Grenz}p_k\coloneqq \left\{\begin{array}{ll}0&\text{for }k\leq m-1, \\ \left(2+\frac{\delta}{m}\right) \frac{\Gamma(k+\delta)\Gamma(m+2+\delta+\frac{\delta}{m})}{\Gamma(m+\delta)\Gamma(k+3+\delta+\frac{\delta}{m})}& \text{for }k\geq m, \end{array}\right.
\end{equation}
where
$\Gamma(t)=\int_0^\infty x^{t-1}e^{-x}dx, \ t>0$ is the Gamma function. For the degree distribution it then holds
\[P_k(n)\overset{\PX}{\to}p_k, (n\to \infty).\]
This convergence is shown in \cite{Hofstad}. The limit of the degree distributions is again a probability distribution by
\begin{thm}
The limit distribution $(p_k)_{k\geq m}$ is a probability distribution.
\end{thm}
\begin{proof} See \cite{Hofstad}.
\end{proof}
Stirling's formula states
\[\frac{\Gamma(x+a)}{\Gamma(x)}\approx x^a.\]
By that we get for $k$ sufficiently big
\begin{eqnarray*}
p_k&=&\left(2+\frac{\delta}{m}\right)\frac{\Gamma(k+\delta)\Gamma(m+2+\delta+\frac{\delta}{m})}{\Gamma(m+\delta)\Gamma(k+3+\delta+\frac{\delta}{m})}\\
&=& \left(2+\frac{\delta}{m}\right)\frac{\Gamma(m+2+\delta+\frac{\delta}{m})}{\Gamma(m+\delta)} \frac{\Gamma(k+\delta)}{\Gamma(k+\delta+3+\frac{\delta}{m})}\\
&\approx& \left(2+\frac{\delta}{m}\right)\frac{\Gamma(m+2+\delta+\frac{\delta}{m})}{\Gamma(m+\delta)} k^{-(3+\frac{\delta}{m})}\\
&=&c_{m,\delta}k^{-\tau},
\end{eqnarray*}
where $c_{m,\delta}=(2+\frac{\delta}{m})\frac{\Gamma(m+2+\delta+\frac{\delta}{m})}{\Gamma(m+\delta)}$ and $\tau=3+\frac{\delta}{m}$. Hence we get that the \PA \ graph is scale free if the number of nodes is large.\\
Since we can not estimate the exact degree distribution we now look at the expectation and variance of the degrees of fixed nodes.
\begin{thm} \label{ewdegstz}
Let $m\geq1,\ \delta>-m$, then for the expected degree of node $i\in[n]$ it holds
\[\E(D^{n}(i)+\delta)=(m+\1_{\{i=1\}}m+\delta)\frac{\Gamma(n+\frac{m}{2m+\delta})\Gamma(i)}{\Gamma(i+\frac{m}{2m+\delta})\Gamma(n)}.\] 
\end{thm}
\begin{proof}
Let $m\geq 1$ and $\delta>-m$ be fixed then
\begin{eqnarray*}
\E(D^{n}(i)+\delta\vert D^{n-1}(i))&=&D^{n-1}(i)+\delta+\E(D^{n}(i)-D^{n-1}(i)\vert D^{n-1}(i))\\
&=&D^{n-1}(i)+\delta+m\PX(\{n,i\}\in E_n \vert PA^{(m,\delta)}_{n-1})\\
&=&(D^{n-1}(i)+\delta)\frac{(2m+\delta)(n-1+\frac{m}{2m+\delta})}{(2m+\delta)(n-1)}\\
&=&(D^{n-1}(i)+\delta)\frac{n-1+\frac{m}{2m+\delta}}{n-1},\\
\end{eqnarray*}
and obviously
\[\E(D^{i}(i)+\delta)=m+\1_{\{i=1\}}m+\delta \ \text{ for all } i\geq 1.\]
For $\E(D^{n}(i)+\delta)$ it then follows recursively
\begin{eqnarray*}
\E(D^{n}(i)+\delta)&=&\E(\E(D^{n}(i)+\delta\vert D^{n-1}(i)))\\
&=&\E(D^{n-1}(i)+\delta)\frac{n-1+\frac{m}{2m+\delta}}{n-1}\\
&& \vdots\\
&=&\E(D^{i}(i)+\delta)\frac{n-1+\frac{m}{2m+\delta}}{n-1}\cdot \ldots \cdot \frac{i+\frac{m}{2m+\delta}}{i}\\
&=&(m+\1_{\{i=1\}}m+\delta)\frac{\Gamma(n+\frac{m}{2m+\delta})\Gamma(i)}{\Gamma(i+\frac{m}{2m+\delta})\Gamma(n)}.
\end{eqnarray*}
\end{proof}

\begin{thm}
Let $m\geq 1,\ \delta >-m$, then for the variance of the degree of a fixed node $i$ at time $n$ it holds
\begin{eqnarray*}
\Var(D^{n}(i))&=&(m+\1_{\{i=1\}}m+\delta)^2\left[\prod_{j=i}^{n-1}(d_j-c_j)-\prod_{j=i}^{n-1}d_j\right]\\ 
&&+\sum_{j=i}^{n-1}\E(D^{j}(i)+\delta)\sqrt{mc_{j}}\prod_{k=j+1}^{n-1}(d_k-c_k),
\end{eqnarray*}
where
\[c_j= \frac{m}{(2m+\delta)^2j^2}\ \text{and} \ d_j=\left(1+\frac{m}{(2m+\delta)j}\right)^2, \forall 1\leq j<n.\]
\end{thm}

\begin{proof}
Let $n\in \N$ and $i\leq n$ as well as $m\geq 1$ and $\delta>-m$ be fixed. To calculate the variance we use the well known identity
\[\Var(X)=\E(\Var(X\vert Y))+\Var(\E(X\vert Y)),\]
where the conditional variance is given by
\[\Var(X\vert Y)\coloneqq \E(X^2\vert Y)-\E(X\vert Y)^2.\]

It then holds for the variance of the degree of node $i$ at time $n$
\begin{eqnarray*}
\Var(D^n(i))&=&\Var(D^n(i)+\delta)\\
&=&\E(\Var(D^n(i)+\delta\vert D^{n-1}(i)))+\Var(\E(D^n(i)+\delta\vert D^{n-1}(i)))\\
&=&\underbrace{\E(\E((D^n(i)+\delta)^2\vert D^{n-1}(i)))}_{I}-\underbrace{\E((\E(D^n(i)+\delta\vert D^{n-1}(i)))^2)}_{II}\\
&&+\underbrace{\Var(\E(D^n(i)+\delta\vert D^{n-1}(i)))}_{III}.\\
\end{eqnarray*}
The different summands can be determined as follows.

(I):
\begin{eqnarray*}
&&\E[\E((D^n(i)+\delta)^2\vert D^{n-1}(i))]\\
&=&\E[\E((D^n(i)+\delta+D^{n-1}(i)-D^{n-1}(i))^2\vert D^{n-1}(i))]\\
&=&\E[(D^{n-1}(i)+\delta)^2]\\
&&+2\E\left[(D^{n-1}(i)+\delta)m\frac{D^{n-1}(i)+\delta}{(2m+\delta)(n-1)}\right] \\
&&+\E[\E((D^n(i)-D^{n-1}(i))^2\vert D^{n-1}(i))]\\
&=&\underbrace{\E[(D^{n-1}(i)+\delta)^2]}_{=\E[\E((D^{n-1}(i)+\delta)^2\vert D^{n-2}(i))]}\left(1+\frac{2m}{(2m+\delta)(n-1)}\right)\\
&&+\underbrace{\E[\E((D^n(i)-D^{n-1}(i))^2\vert D^{n-1}(i))]}_{IV}.
\end{eqnarray*}

(II): With the calculations in the proof of Theorem \ref{ewdegstz} we get
\begin{eqnarray*}
&&\E[\E(D^n(i)+\delta\vert D^{n-1}(i))^2]\\
&=&\E\left[\left((D^{n-1}(i)+\delta)\left(1+\frac{m}{(2m+\delta)(n-1)}\right)\right)^2\right]\\
&=&\underbrace{\E[(D^{n-1}(i)+\delta)^2]}_{=\E[\E((D^{n-1}(i)+\delta)^2\vert D^{n-2}(i))]}\left(1+\frac{m}{(2m+\delta)(n-1)}\right)^2.
\end{eqnarray*}

(III):
\begin{eqnarray*}
&&\Var[\E(D^n(i)+\delta\vert D^{n-1}(i))]\\
&=&\Var\left[(D^{n-1}(i)+\delta)\left(1+\frac{m}{(2m+\delta)(n-1)}\right)\right]\\
&=&\Var[D^{n-1}(i)+\delta]\left(1+\frac{m}{(2m+\delta)(n-1)}\right)^2.
\end{eqnarray*}

(IV): For the calculation of \[\E[\E((D^n(i)-D^{n-1}(i))^2\vert D^{n-1}(i))]\] we use that, by construction, the number of edges between a new node and a node already in the graph is binomially distributed, hence
\[D^n(i)-D^{n-1}(i)\vert D^{n-1}(i) \sim \mathrm{Bin}(m,p), \ \text{where} \ p\coloneqq \frac{D^{n-1}(i)+\delta}{(2m+\delta)(n-1)}.\]
For $X\sim \mathrm{Bin}(m,p)$ one gets $\E(X^2)=mp(1-p)+(mp)^2$ and it follows
\begin{eqnarray*}
&&\E[\E((D^n(i)-D^{n-1}(i))^2\vert D^{n-1}(i))]\\
&=&\E[D^{n-1}(i)+\delta]\frac{m}{(2m+\delta)(n-1)}\\
&&+\underbrace{\E[(D^{n-1}(i)+\delta)^2]}_{=\E[\E((D^{n-1}(i)+\delta)^2\vert D^{n-2}(i))]}\frac{m(m-1)}{(2m+\delta)^2(n-1)^2}.
\end{eqnarray*}

We now introduce the following abbreviations for simplicity
\begin{eqnarray*}
V_n&\coloneqq& \Var(D^n(i)+\delta),\\
E_n&\coloneqq&\E[\E((D^n(i)+\delta)^2\vert D^{n-1}(i))],\\
D_n&\coloneqq&\E[D^{n}(i)+\delta].
\end{eqnarray*}
By putting the results into the first equation of the proof and using the abbreviations we get
\begin{eqnarray*}
V_n&=&E_{n-1}\left(1+\frac{2m}{(2m+\delta)(n-1)}\right)+E_{n-1}\frac{m(m-1)}{(2m+\delta)^2(n-1)^2}\\
&&+D_{n-1}\frac{m}{(2m+\delta)(n-1)}-E_{n-1}\left(1+\frac{m}{(2m+\delta)(n-1)}\right)^2\\
&&+V_{n-1}\left(1+\frac{m}{(2m+\delta)(n-1)}\right)^2\\
&=&-E_{n-1}\frac{m}{(2m+\delta)^2(n-1)^2}+D_{n-1}\frac{m}{(2m+\delta)(n-1)}\\
&&+V_{n-1}\left(1+\frac{m}{(2m+\delta)(n-1)}\right)^2.\\
\end{eqnarray*}
For $j<n$ let
\begin{eqnarray*}
c_j&\coloneqq& \frac{m}{(2m+\delta)^2j^2},\\
d_j&\coloneqq& \left(1+\frac{m}{(2m+\delta)j}\right)^2.
\end{eqnarray*}
Then we can calculate $\Var(D^n(i))$ with the following recursion formulas and termination conditions
\begin{eqnarray*}
V_n&=&E_{n-1}(-c_{n-1})+D_{n-1}\sqrt{mc_{n-1}}+V_{n-1}d_{n-1},\\
E_n&=&E_{n-1}(d_{n-1}-c_{n-1})+D_{n-1}\sqrt{mc_{n-1}},\\
&&\\
V_i&=&\Var(D^i(i)+\delta)=0,\\
E_i&=&\E((D^i(i)+\delta)^2)=(m+\1_{\{i=1\}}m+\delta)^2.\\
\end{eqnarray*}
Two steps of the recursion yield
\begin{eqnarray*}
V_n&=&E_{n-1}(-c_{n-1})+D_{n-1}\sqrt{mc_{n-1}}+V_{n-1}d_{n-1}\\
&=&E_{n-2}(c_{n-1}c_{n-2}-c_{n-1}d_{n-2}-c_{n-2}d_{n-1})+D_{n-1}\sqrt{mc_{n-1}}\\
&&+D_{n-2}\sqrt{mc_{n-2}}(d_{n-1}-c_{n-1})+V_{n-2}d_{n-1}d_{n-2}\\
&=&E_{n-2}[(d_{n-1}-c_{n-1})(d_{n-2}-c_{n-2})-d_{n-1}d_{n-2}]+D_{n-1}\sqrt{mc_{n-1}}\\
&&+D_{n-2}\sqrt{mc_{n-2}}(d_{n-1}-c_{n-1})+V_{n-2}d_{n-1}d_{n-2}\\
&=&E_{n-3}(d_{n-3}-c_{n-3})[(d_{n-1}-c_{n-1})(d_{n-2}-c_{n-2})-d_{n-1}d_{n-2}]\\
&&+D_{n-3}\sqrt{mc_{n-3}}[(d_{n-1}-c_{n-1})(d_{n-2}-c_{n-2})-d_{n-1}d_{n-2}]\\
&&+D_{n-1}\sqrt{mc_{n-1}}+D_{n-2}\sqrt{mc_{n-2}}(d_{n-1}-c_{n-1})-E_{n-3}c_{n-3}d_{n-1}d_{n-2}\\
&&+D_{n-3}\sqrt{mc_{n-3}}d_{n-1}d_{n-2}+V_{n-3}d_{n-1}d_{n-2}d_{n-3}\\
&=&E_{n-3}[(d_{n-3}-c_{n-3})(d_{n-1}-c_{n-1})(d_{n-2}-c_{n-2})-d_{n-1}d_{n-2}d_{n-3}]\\
&&+D_{n-1}\sqrt{mc_{n-1}}+D_{n-2}\sqrt{mc_{n-2}}(d_{n-1}-c_{n-1})\\
&&+D_{n-3}\sqrt{mc_{n-3}}(d_{n-1}-c_{n-1})(d_{n-2}-c_{n-2})+V_{n-3}d_{n-1}d_{n-2}d_{n-3}.
\end{eqnarray*}
Hence it follows
\begin{eqnarray*}
V_n&=&V_i\prod_{j=i}^{n-1}d_j+E_i\left[\prod_{j=i}^{n-1}(d_j-c_j)-\prod_{j=i}^{n-1}d_j\right]\\
&&+\sum_{j=i}^{n-1}D_j\sqrt{mc_{j}}\prod_{k=j+1}^{n-1}(d_k-c_k)\\
&=&(m+\1_{\{i=1\}}m+\delta)^2\left[\prod_{j=i}^{n-1}(d_j-c_j)-\prod_{j=i}^{n-1}d_j\right]\\
&& +\sum_{j=i}^{n-1}D_j\sqrt{mc_{j}}\prod_{k=j+1}^{n-1}(d_k-c_k).
\end{eqnarray*}
\end{proof}
\begin{figure}[h]
\centering
\includegraphics[width=11cm]{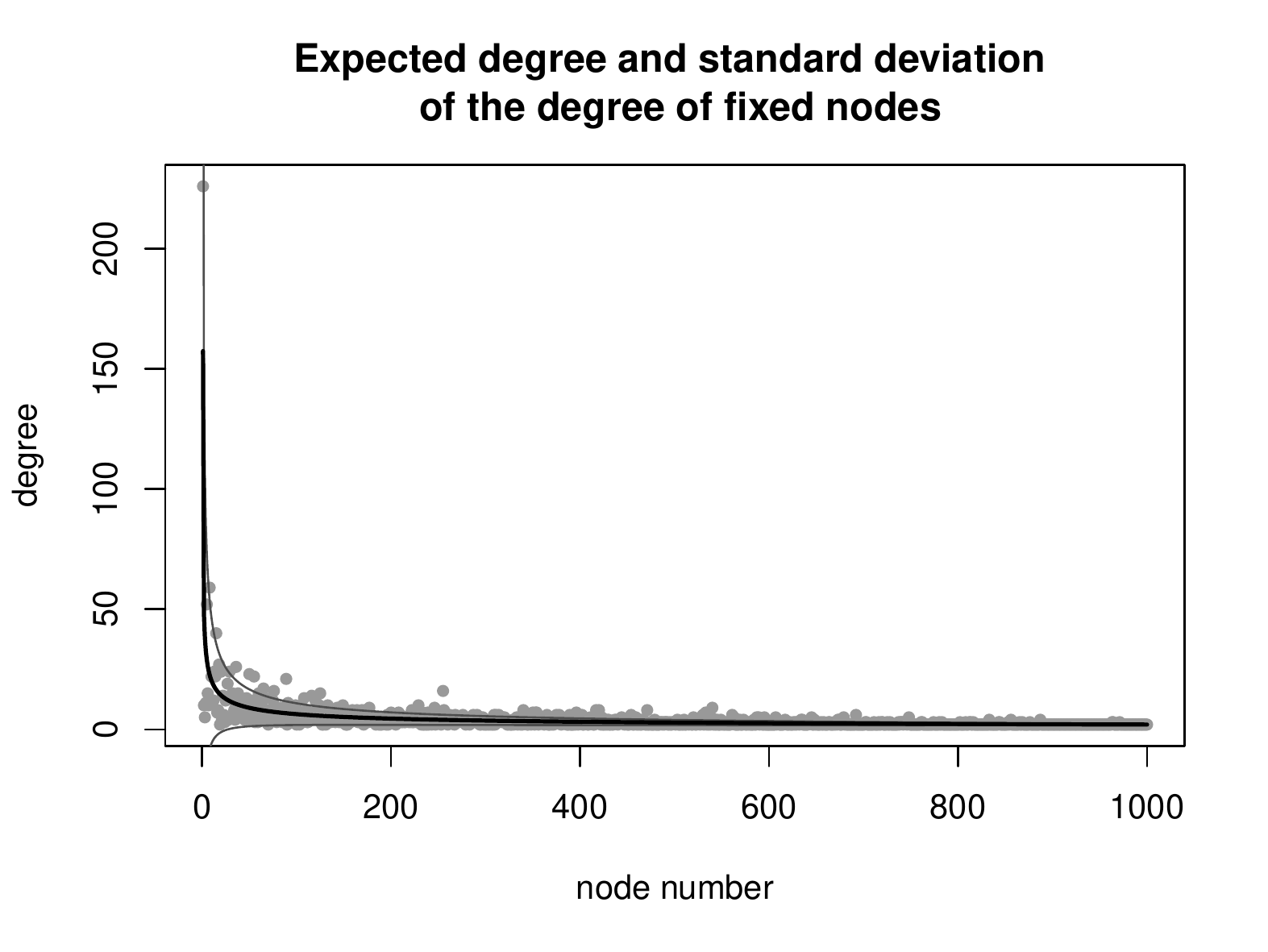}
\caption{The points are the degrees of the nodes in the order in which they were added to the graph. The thick black line is the expected degree of the nodes of a \PA \ graph with $n=1000$ nodes, $m=2$ and $\delta=0$. The thin lines are the standard deviation from the mean of the degrees.}\label{ewdeg}
\end{figure}
In figure \ref{ewdeg} the degrees are plotted with their expectations and variances.

\subsubsection{Subgraphs of \PA\ graphs}
Since we can not determine the degree distribution of the \PA\ graphs, we will analyze them numerically. The constructions of the subgraphs are the same as for the \ER\ graphs.

We first look at the mean degree of subgraphs of \PA \ graphs. For the construction of the subgraph we decide for every edge if it is deleted while keeping all nodes. Let therefore the probability for an edge of the graph to be kept in the subgraph be $q\in (0,1)$. It then holds for fixed $i\in [n]$
\[\PX(D_{Sub}^n(i)=l\vert D^n(i)=k)={k \choose l} q^l(1-q)^{k-l} \ \text{for} \ l\in\{0,\ldots,k\},\ k\leq n.\]
With that we can determine the degree distribution of the subgraph given the degree distribution of the graph by
\begin{eqnarray*}
\PX(D^{n}_{Sub}(i)=l)&=&\sum_{k=l}^{n-1}\PX(D^n_{Sub}(i)=l\vert D^n(i)=k)\PX(D^n(i)=k)\\
&=&\sum_{k=l}^{n-1}{k \choose l}q^l(1-q)^{k-l}\frac{1}{n}\sum_{i=1}^n\1_{\{D^n(i)=k\}}\\
&=&\frac{1}{n}\sum_{i=1}^n{D^n(i) \choose l}q^l(1-q)^{D^n(i)-l},
\end{eqnarray*}
where $D^n(i)\sim P_k(n)$. Hence for the mean degree it holds
\begin{eqnarray*}
\E(D^n_{Sub}(i))&=& \sum_{l=0}^{n-1} l \frac{1}{n}\sum_{i=1}^n{D^n(i) \choose l}q^l(1-q)^{D^n(i)-l}\\
&=& \frac{1}{n}\sum_{i=1}^n\sum_{l=0}^{D^n(i)} l {D^n(i) \choose l}q^l(1-q)^{D^n(i)-l}\\
&=& \frac{1}{n}\sum_{i=1}^nD^n(i)q= \frac{q}{n}2\# E_n=2mq.
\end{eqnarray*}
Since we are not able to determine the exact degree distributions we calculated the empirical degree distributions. These are depicted in figure \ref{PAdeg} for a \PA \ graph and its subgraph by selection of edges together with the limit distribution from formula (\ref{Grenz}). By comparing the two degree distributions one sees that they are both scale free with approximately the same exponent. But the subgraph is not necessarily connected, whereas the \PA \ graph itself is by construction. Hence it is not possible to construct the subgraph with our \PA \ model.

We now want to look at subgraphs of our \PA \ graph which are constructed by deleting nodes. We again select the nodes in the two different ways we used for the \ER\ graphs. Hence we fix a number of nodes and choose a subgraph of that size uniformly at random out of the set of all subgraphs with that size. For the other mechanism we decide for every node independently of the others if we delete it. The empirical degree distributions are depicted in figure \ref{degPA} together with the limit distribution from formula (\ref{Grenz}). One can see that the exponent $\tau$ of the power law stays approximately the same as in the \PA \ graph but the subgraphs are again not necessarily connected which is a complex problem related to percolation theory. Additionally the mean degree does not have to be a multiplicity of 2 which is the case in our model. Hence its not possible to construct these subgraphs with our initial \PA \ model used to construct the entire graph.
\begin{figure}[p]
\centering
\includegraphics[width=6cm]{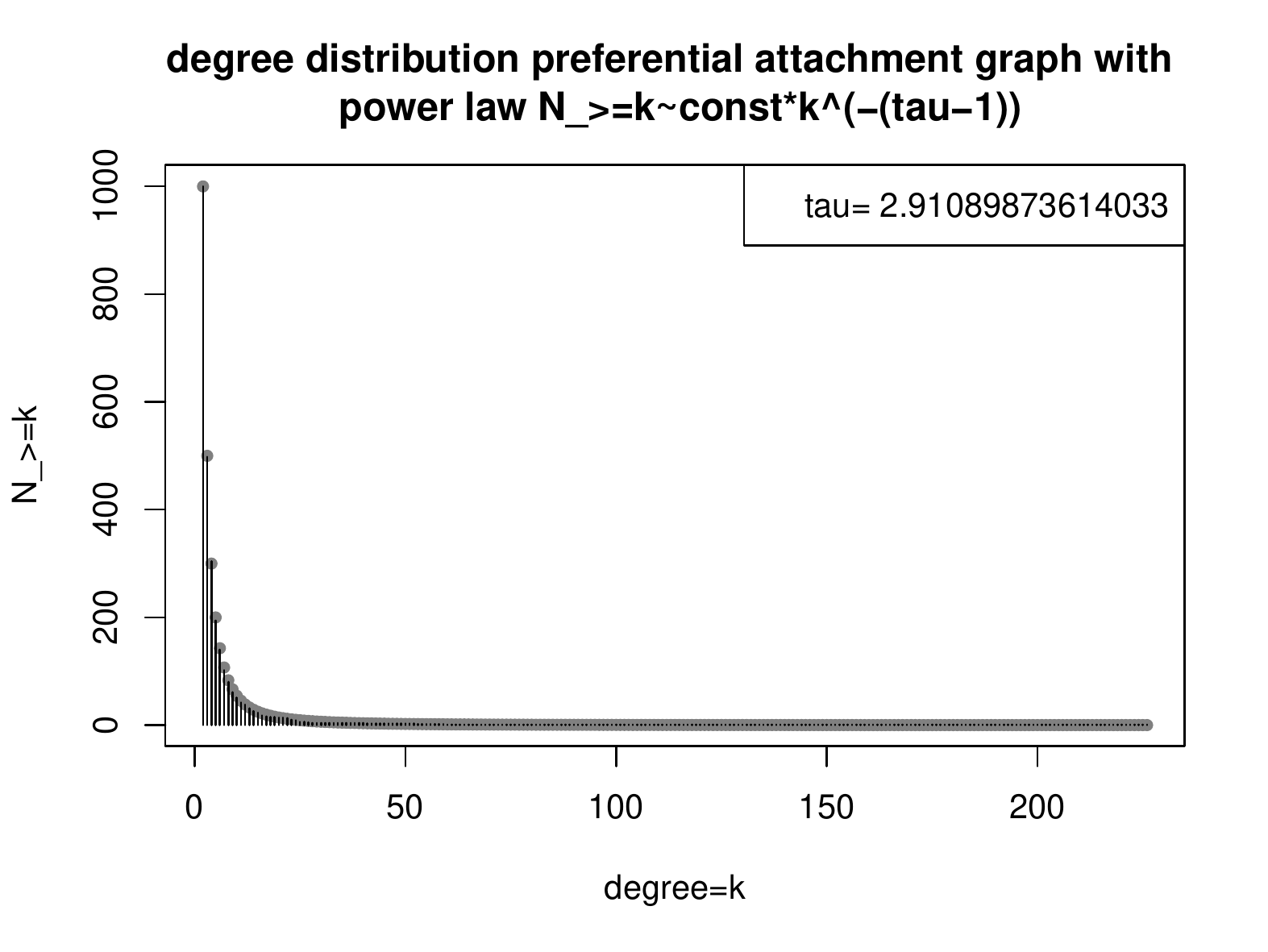}
\includegraphics[width=6cm]{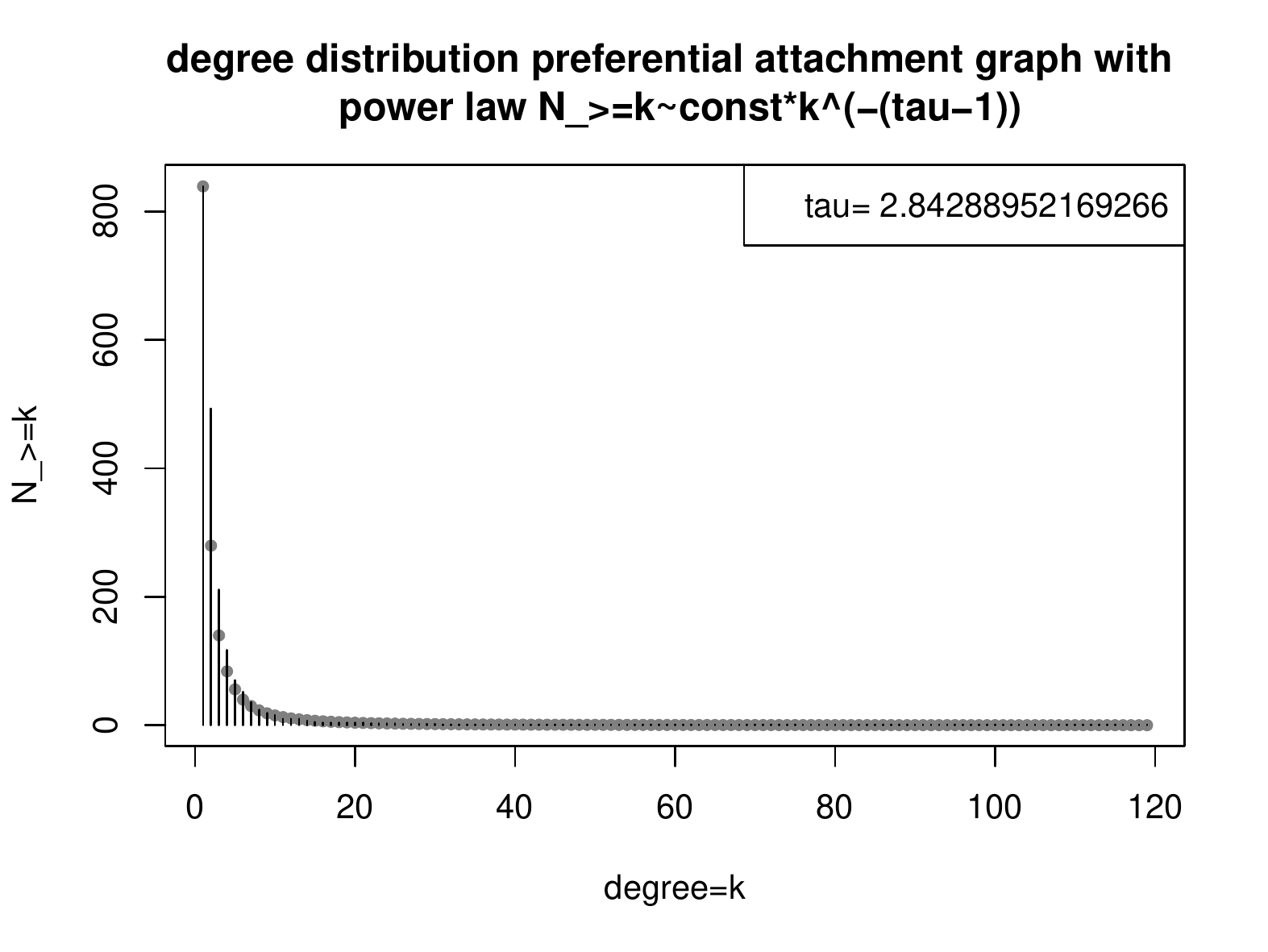}
\includegraphics[width=6cm]{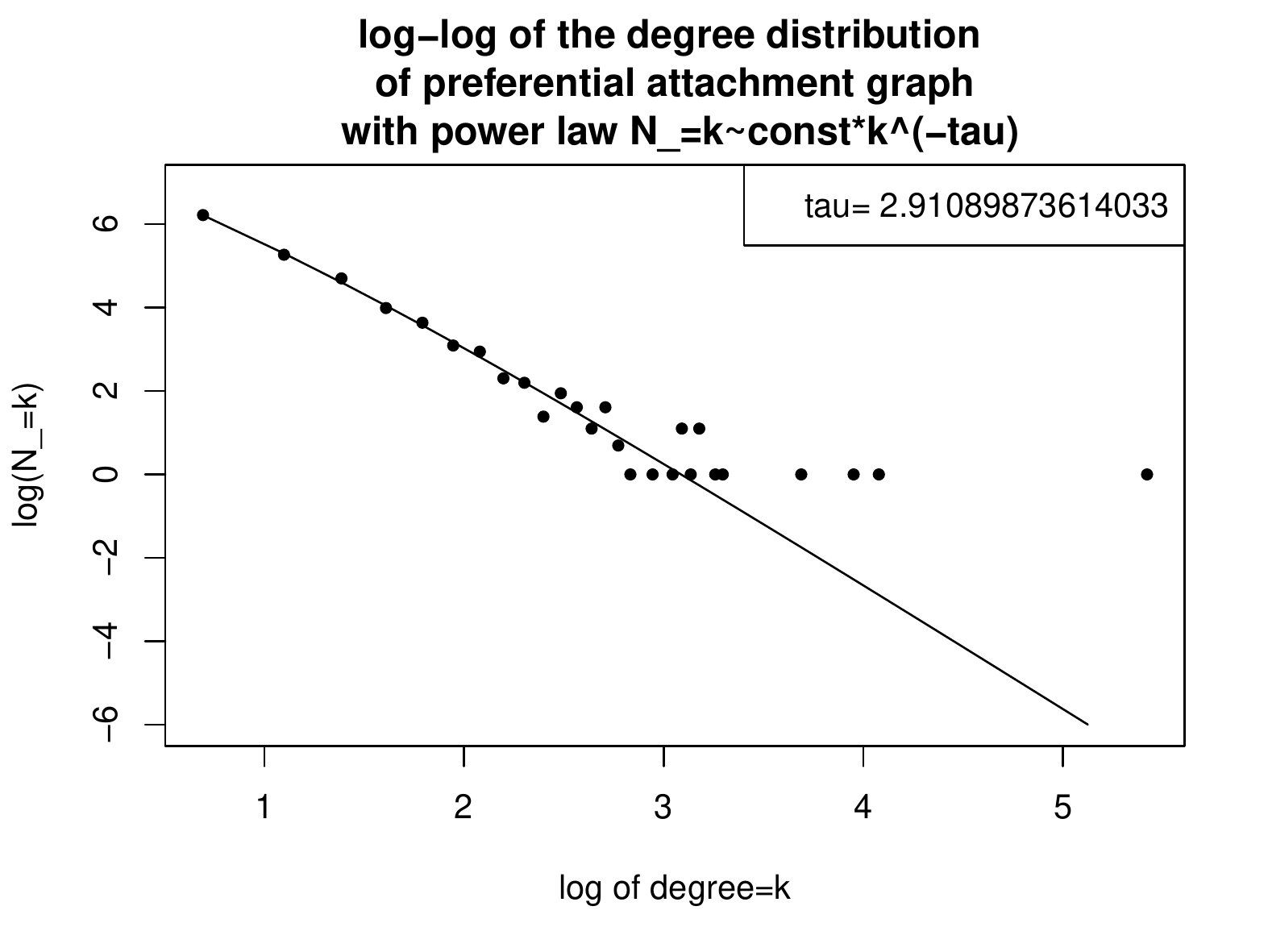}
\includegraphics[width=6cm]{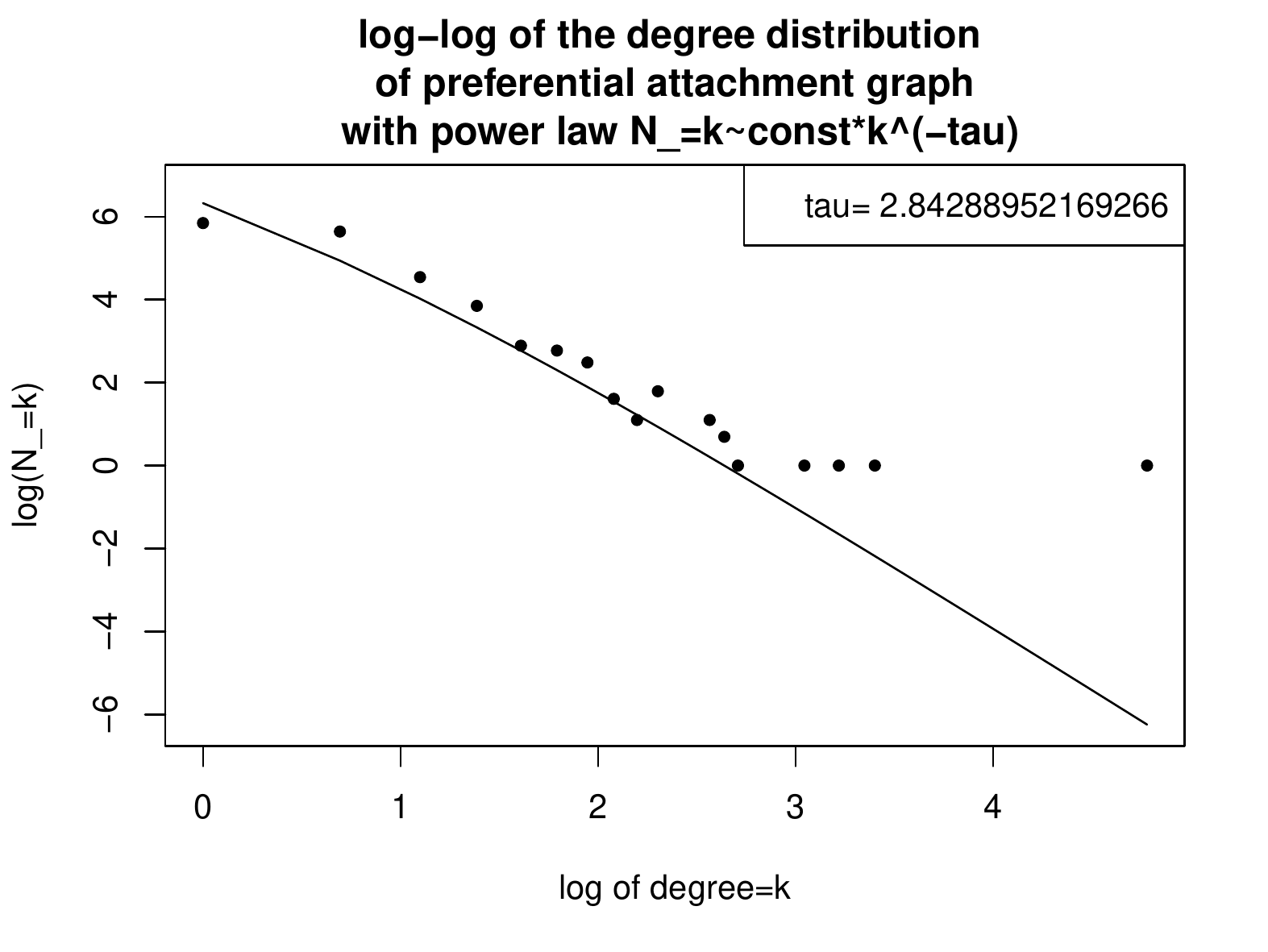}
\includegraphics[width=6cm]{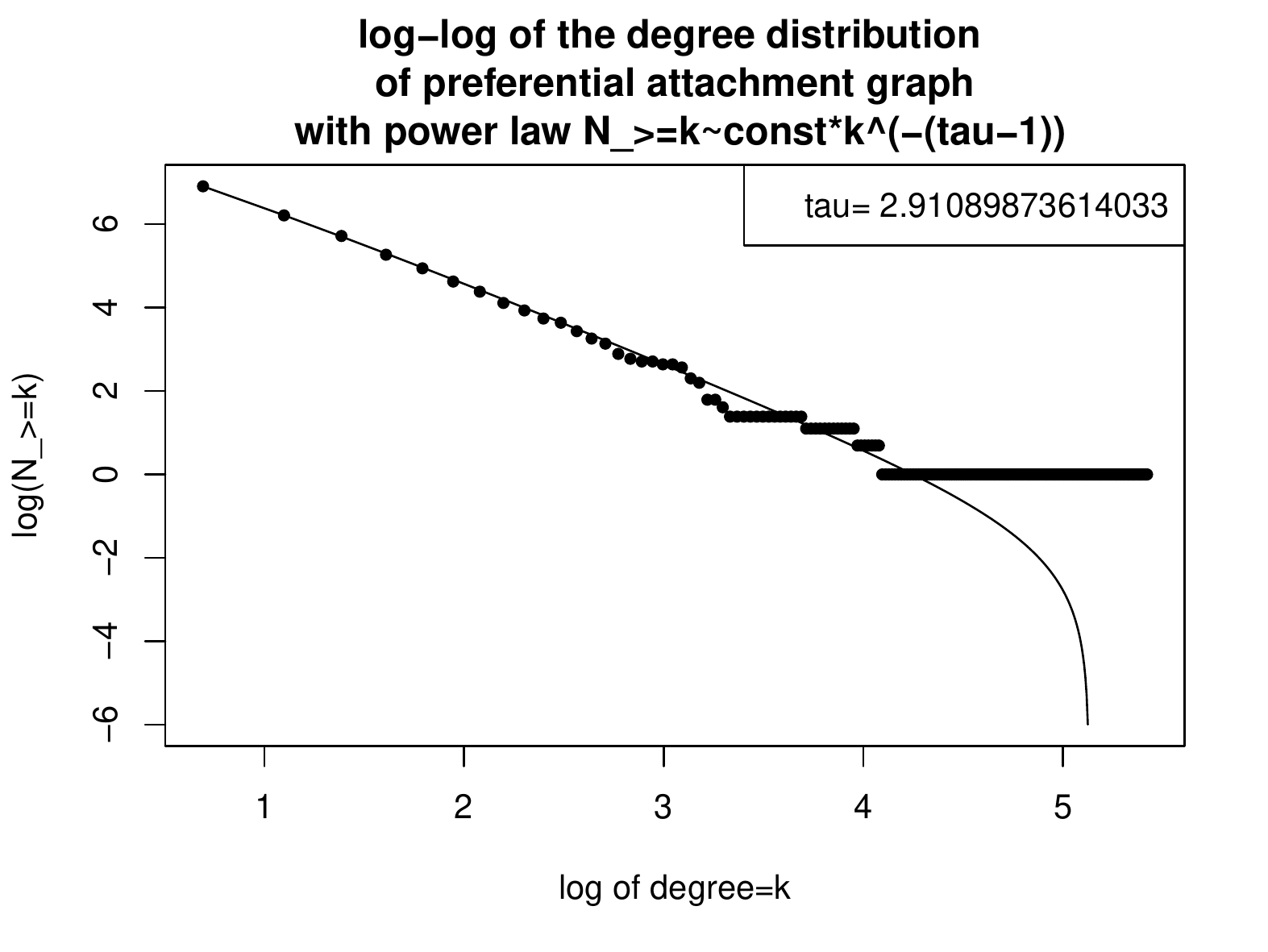}
\includegraphics[width=6cm]{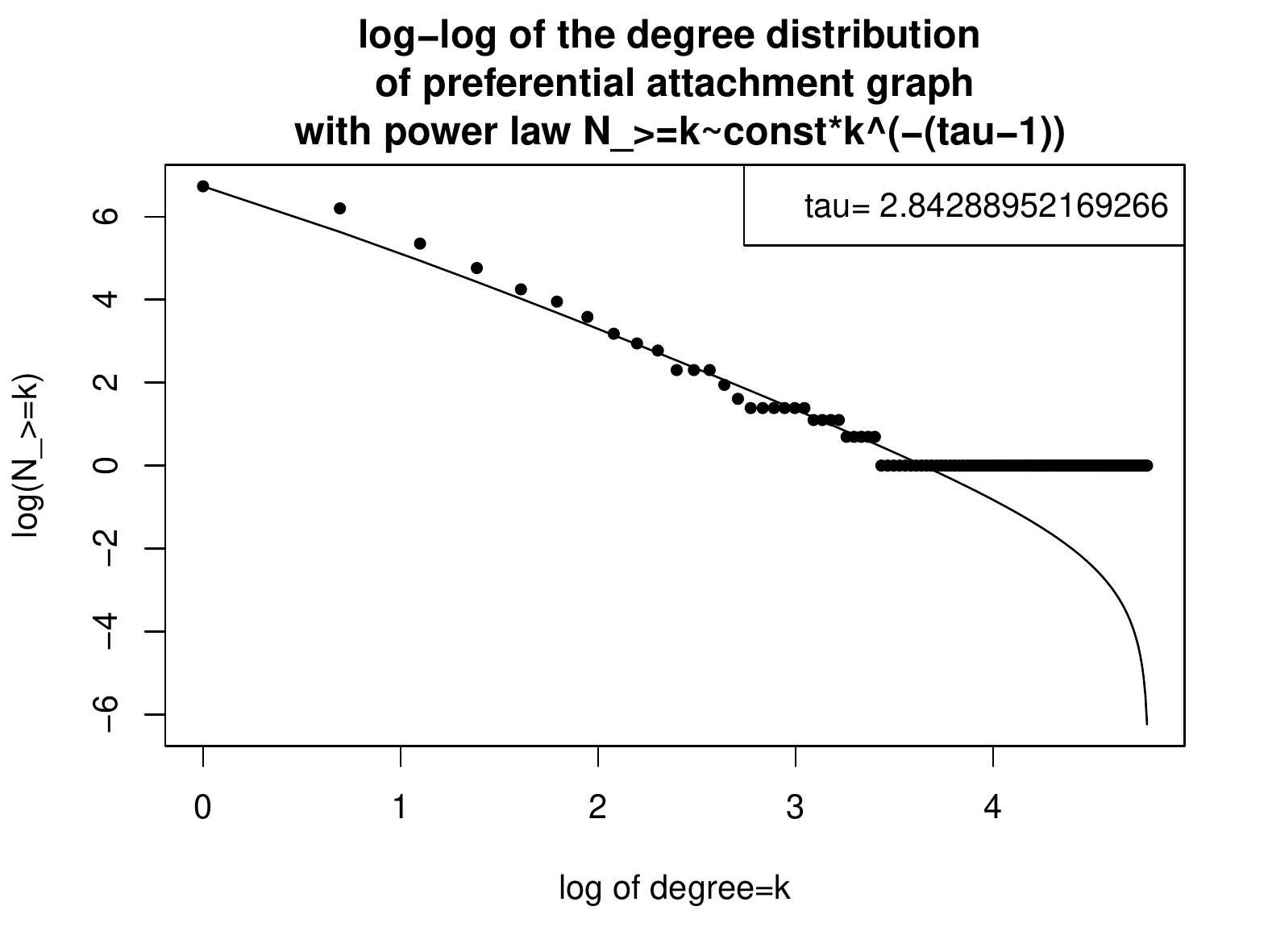}
\caption{On the left is the degree distribution of a PA graph with $n=1000$, $m=2$ and $\delta=0$. This PA graph is used for the construction of all PA subgraphs and the analysis of the expected degrees.\newline
On the right is the degree distribution of a subgraph constructed by selection of edges with probability $p=0.5$ and limit distribution $(p_k)_{k\geq m}$ with $n=500$, $m=1$ and $\delta=0$.\newline
The dots in the top pictures represent the limit distribution $(p_k)_{k\geq m}$. In the lower four pictures $(p_k)_{k\geq m}$ is represented by the line.
The \PA \ graph is scale free with $\tau=2.911$, the subgraph has power law exponent $\tau=2.843$.}\label{PAdeg}
\end{figure}
\begin{figure}[p]
\centering
\includegraphics[width=6cm]{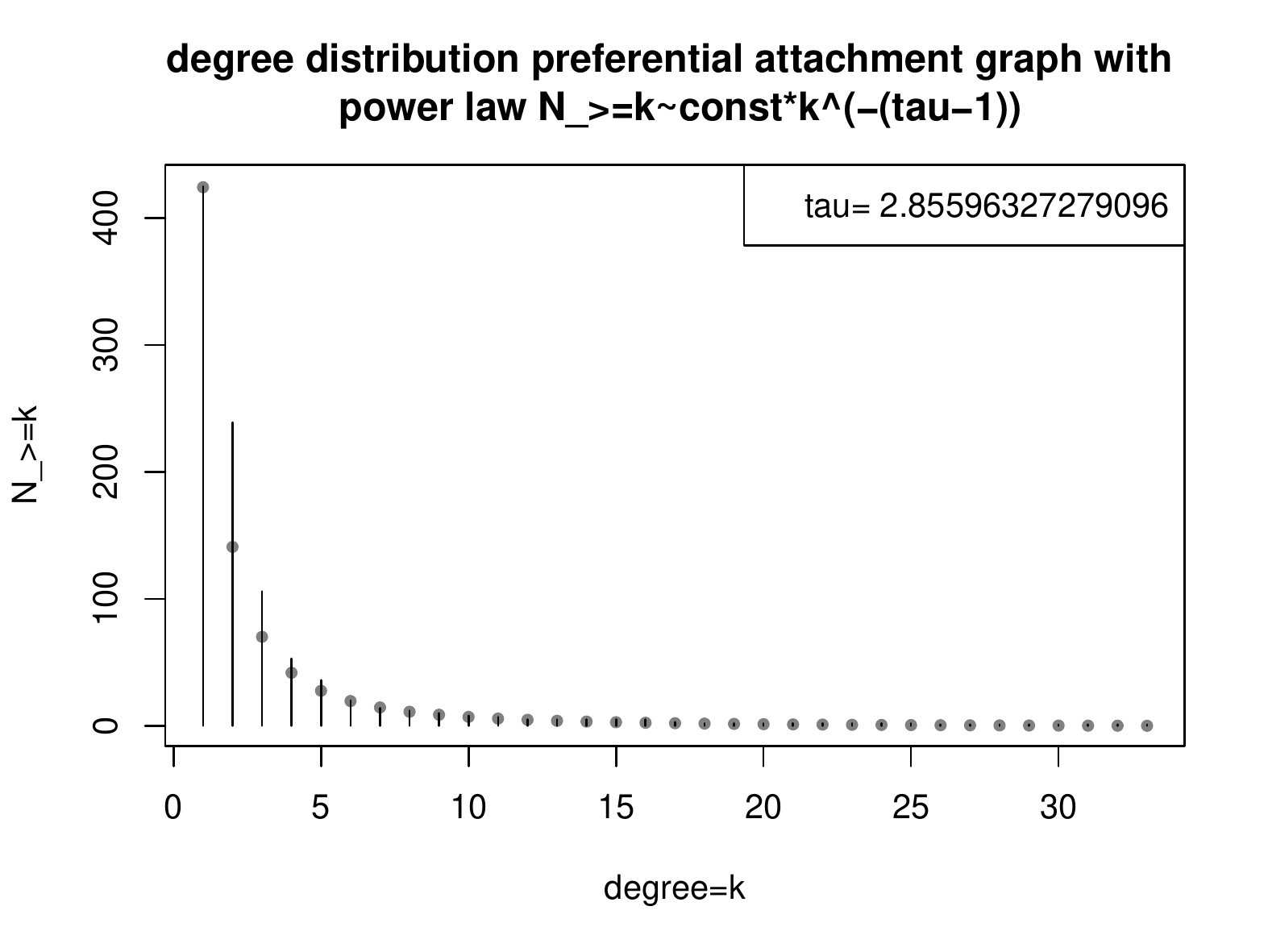}
\includegraphics[width=6cm]{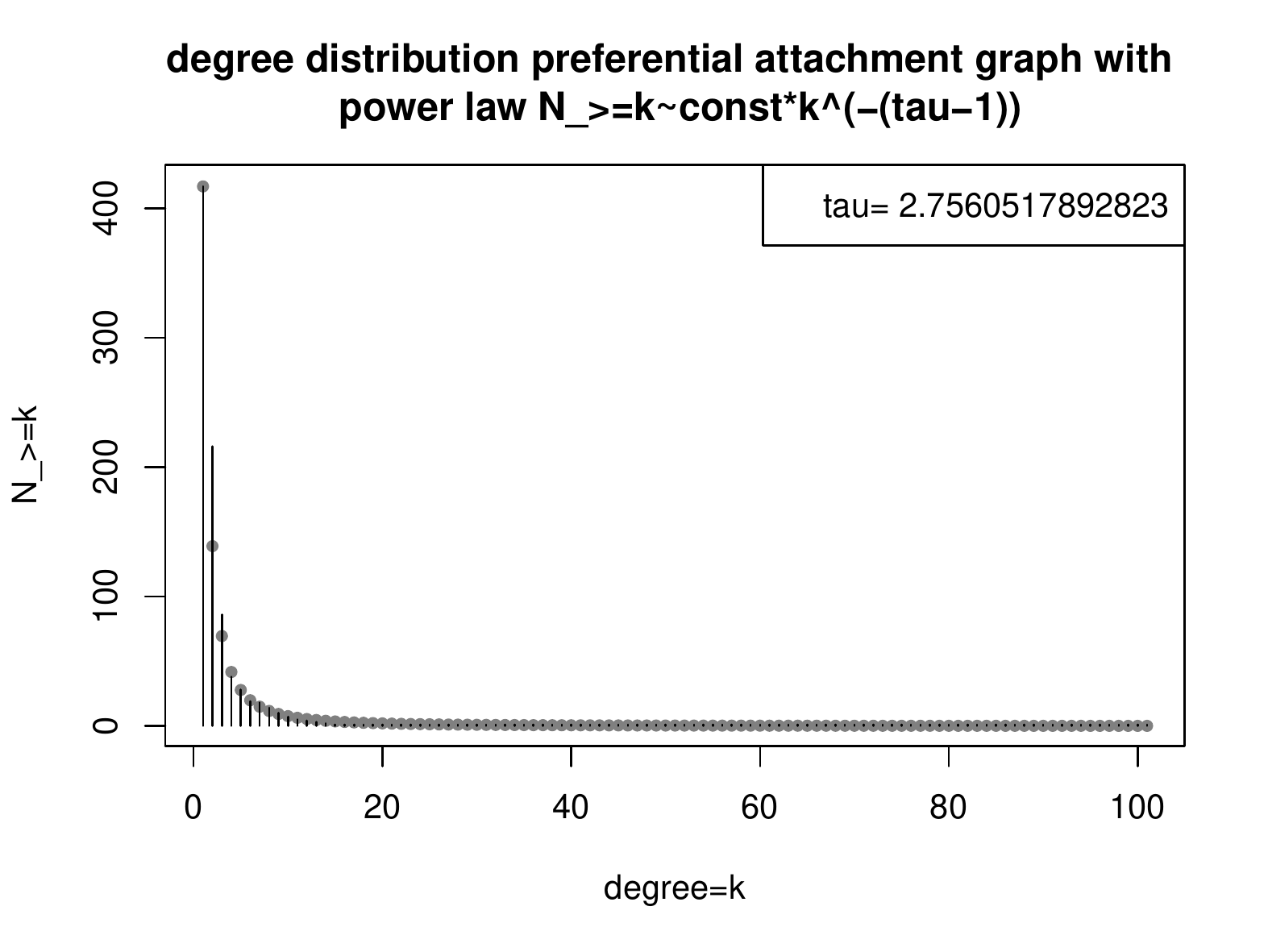}
\includegraphics[width=6cm]{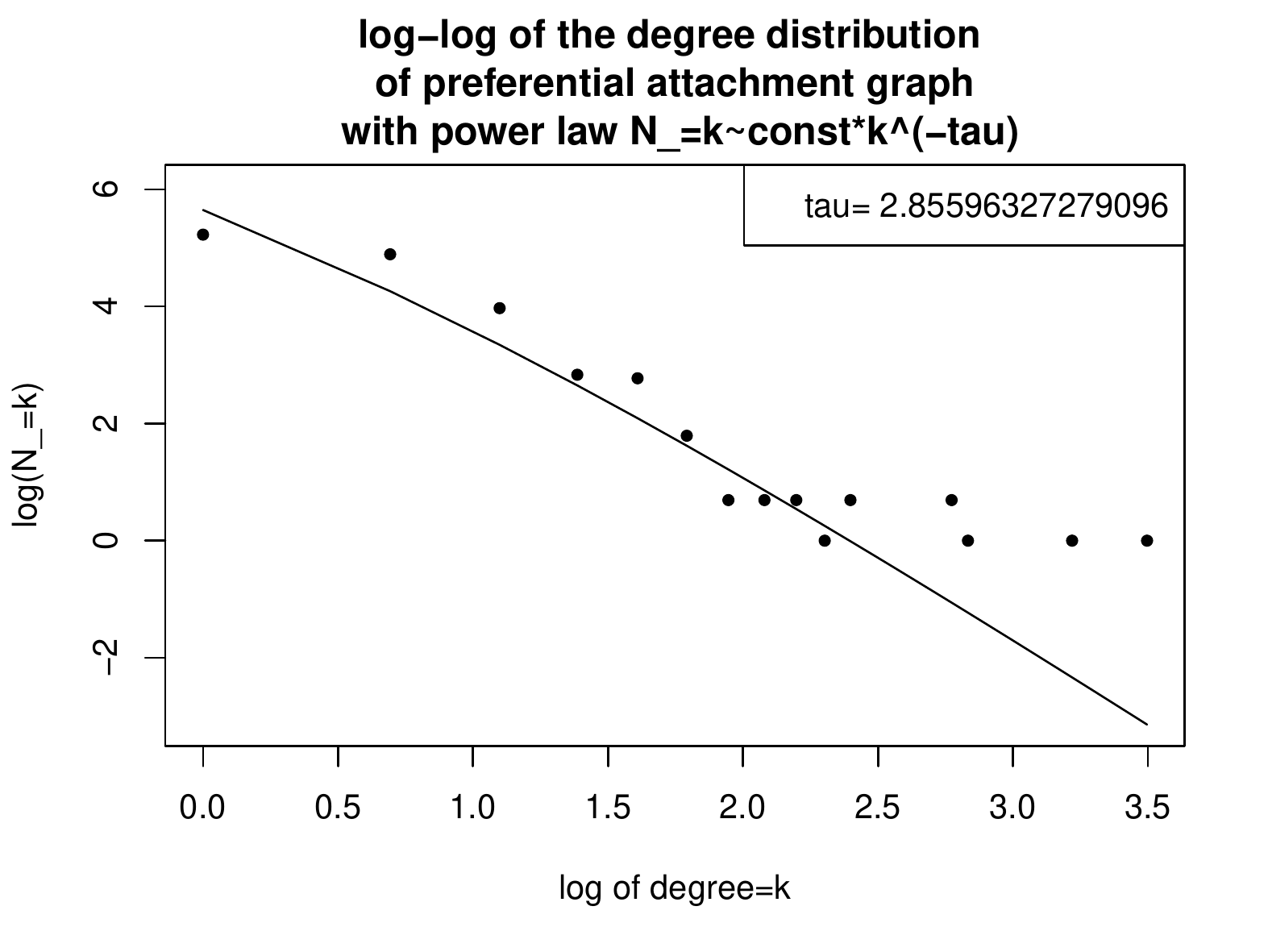}
\includegraphics[width=6cm]{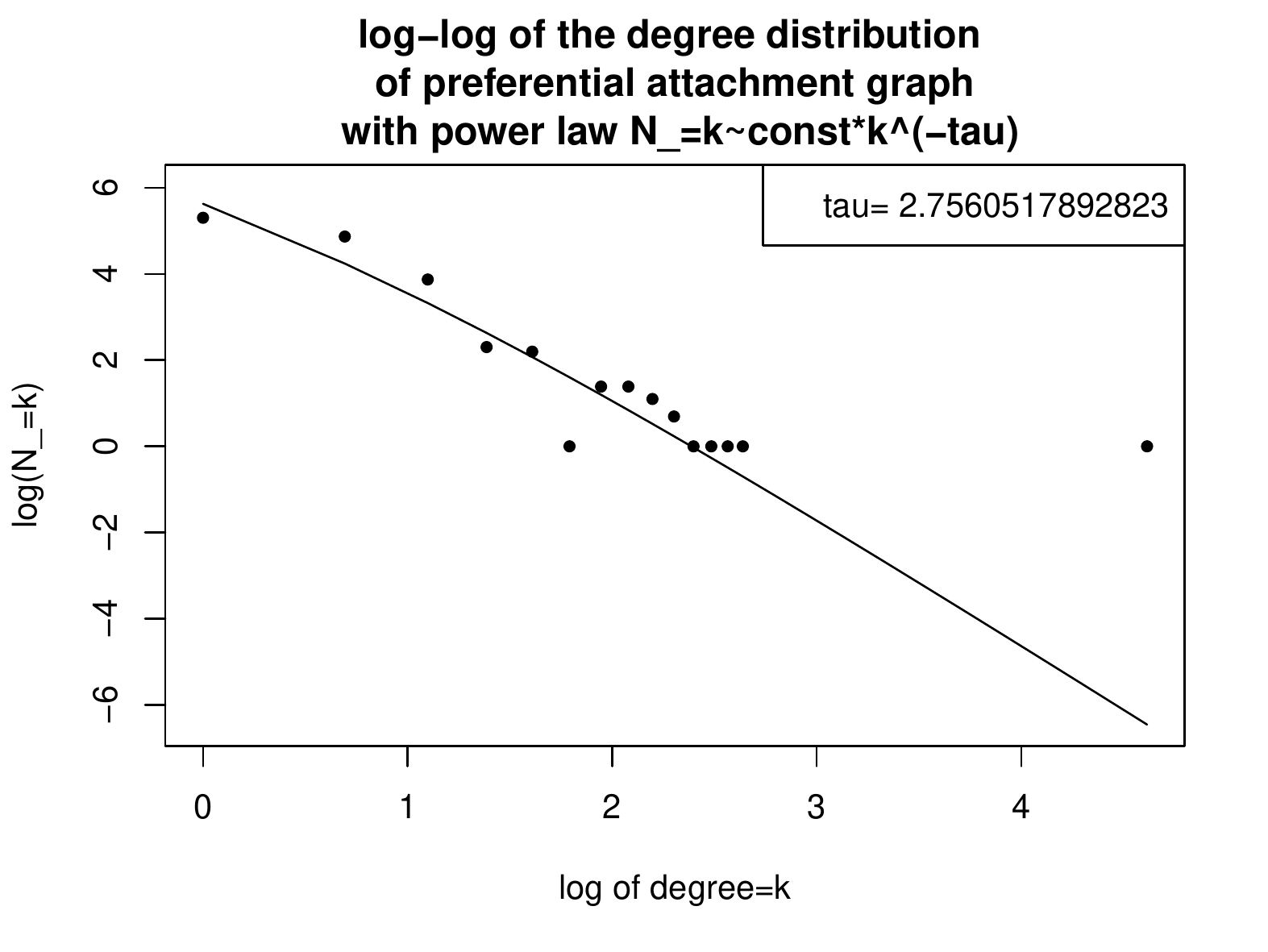}
\includegraphics[width=6cm]{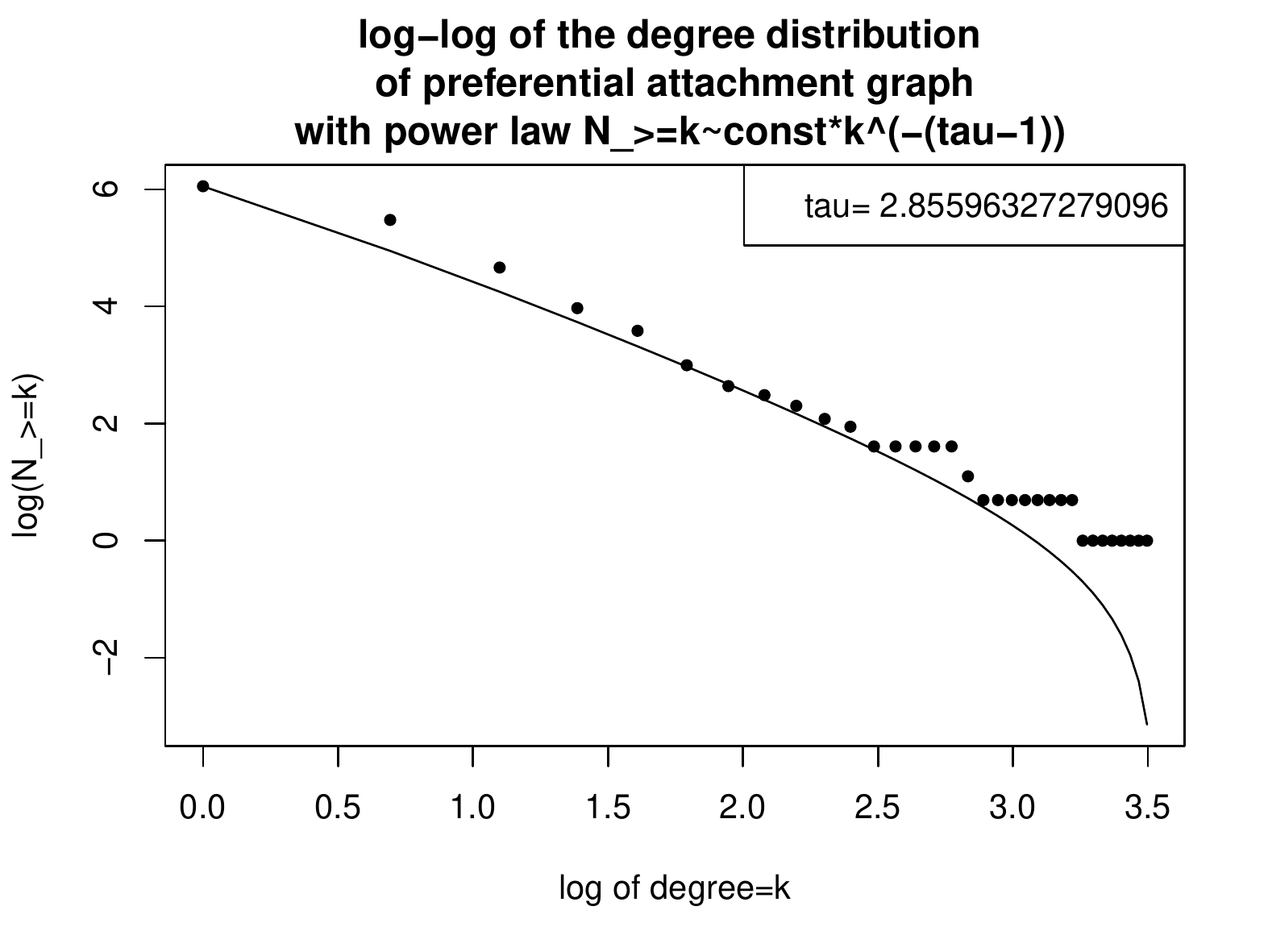}
\includegraphics[width=6cm]{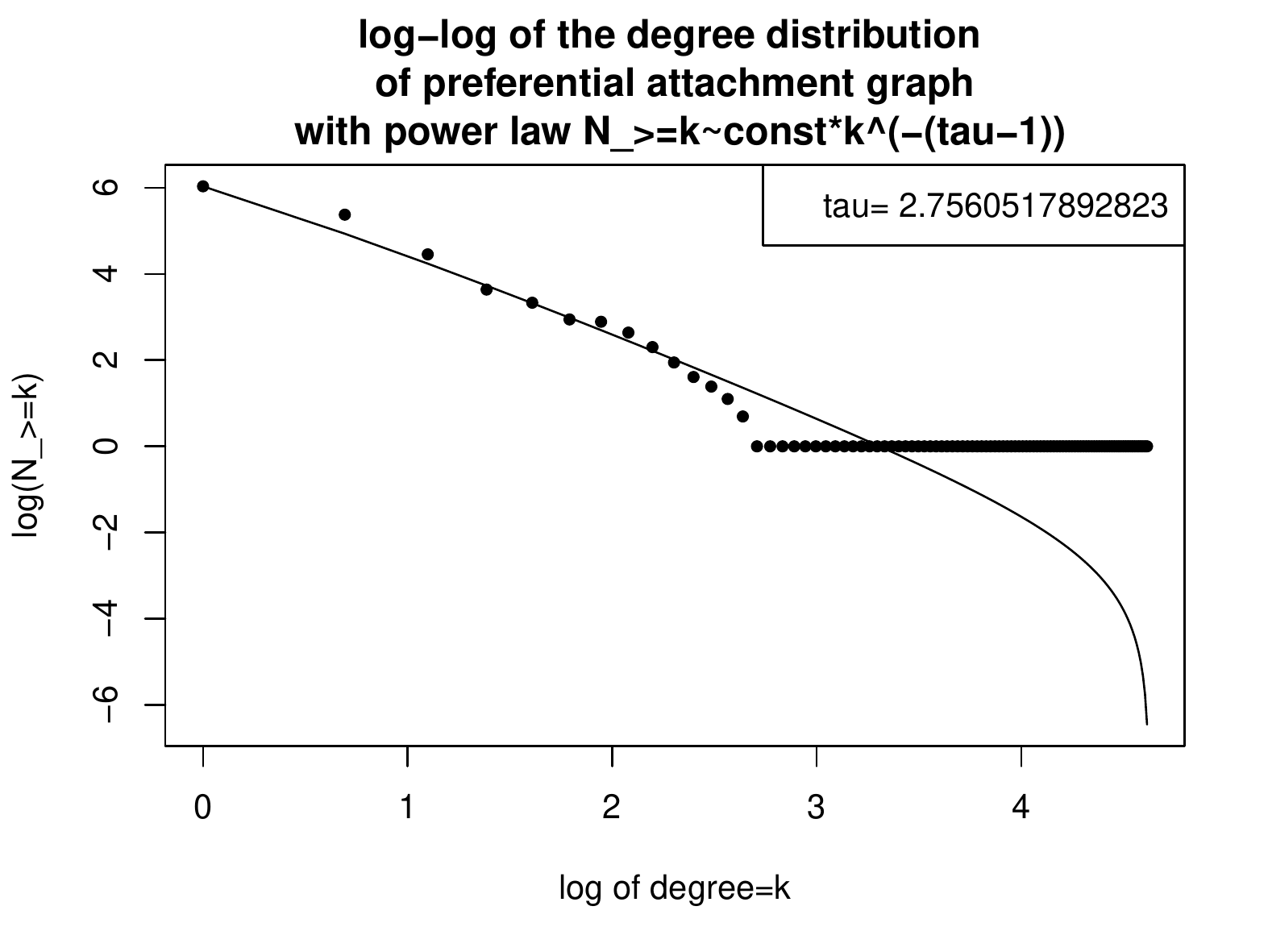}
\caption{Here we see the degree distributions of two subgraphs constructed by selection of nodes together with the corresponding limit distributions. On the left we have uniform selection of nodes with $500$ nodes and on the right binomial selection of nodes with probability $q=0.5$. The parameters of the limit distributions are in both cases given by $n=500$, $m=1$ and $\delta=0$. Uniform selection of nodes yields a power law exponent $\tau=2.847$, whereas the binomial selected subgraph is scale free with $\tau=2.756$.\newline
We calculated $m$ for the subgraphs by the identity $\#E=mn$.\newline
One can see that the limit distribution is a good approximation for the degree distribution in the beginning. The large deviation at the end is due to the small number of nodes and single nodes having a degree which occurs with a very small probability.}\label{degPA}
\end{figure}


\section{Protein networks}\label{proteins}
\begin{figure}[h]
\centering
\includegraphics[width=12cm]{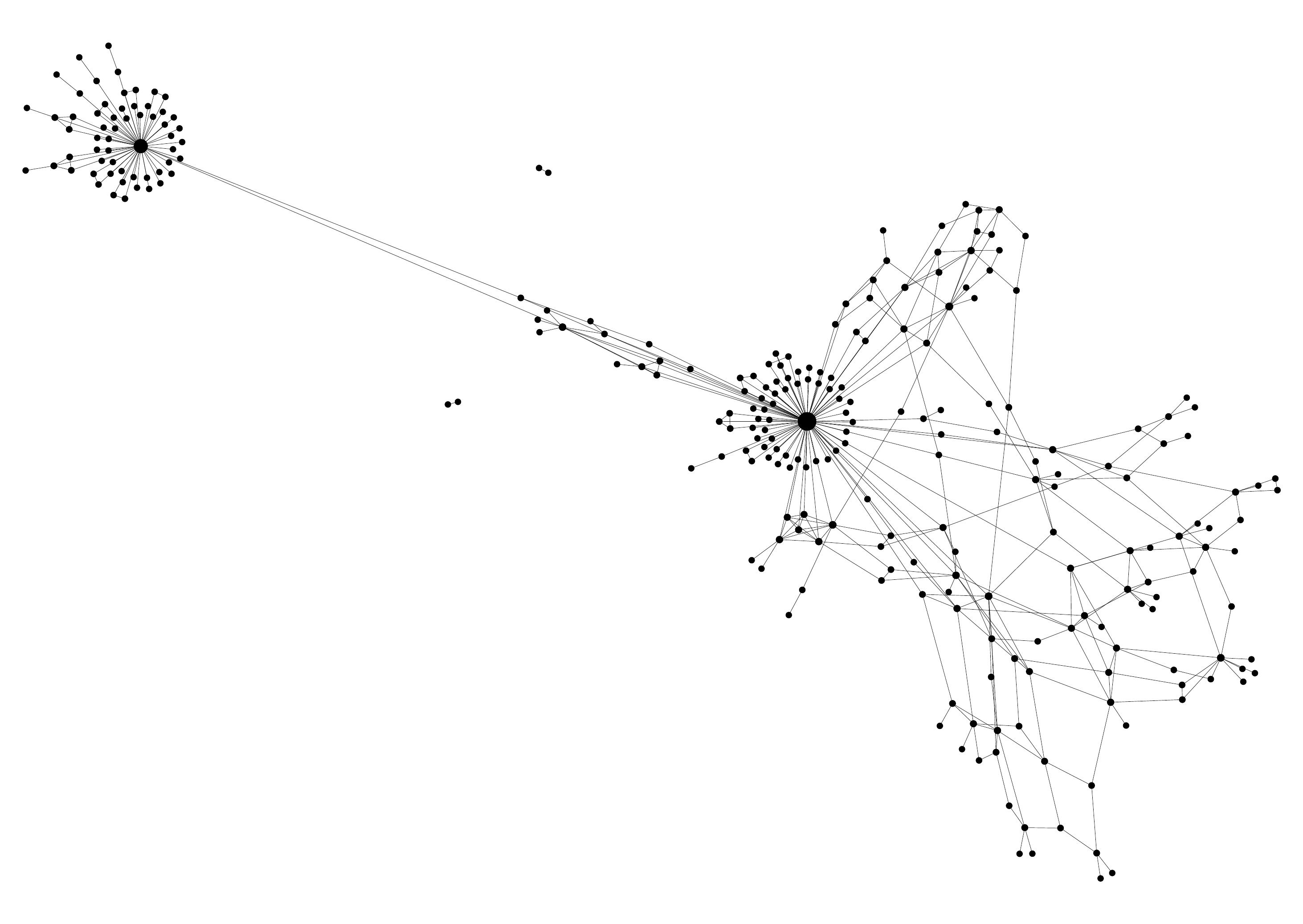}
\caption{Graphical illustration of a protein network of 263 class A $\beta$-lactamases of the TEM family where adjacent proteins differ by one mutation.}\label{ProtNW}
\end{figure}

We now want to analyze a protein network in order to decide if it can be constructed with our \PA\ model. Proteins are sequences of amino acids. The length of most of them is between 250 and 420 amino acids. But there are some with lengths up to 27000 amino acids.\\
Proteins consist of up to 20 different amino acids. The order of the amino acids determines the structure and function of the protein. It is obvious that not every constellation of amino acids yields  a biological sufficient protein. Yet proteins do not necessarily loose their functionality if only few of the amino acids are replaced.\\
It is possible to compare proteins regarding the order of their amino acids and assign scores according to their similarity. One could also look at the number of mutations between two proteins as a score.\\
Proteins of the same family can then be connected to form a network by connecting two proteins that differ by one mutation. The structure of such a network is depicted in figure \ref{ProtNW}. We therefore took class A $\beta$-lactamases of the TEM family which have a length between $200$ and $300$ amino acids. Figure \ref{deg_Prot} shows the degree distribution of the protein network which is scale free with $\tau=2.641 $.
\begin{figure}[ht]
\centering
\includegraphics[width=6cm]{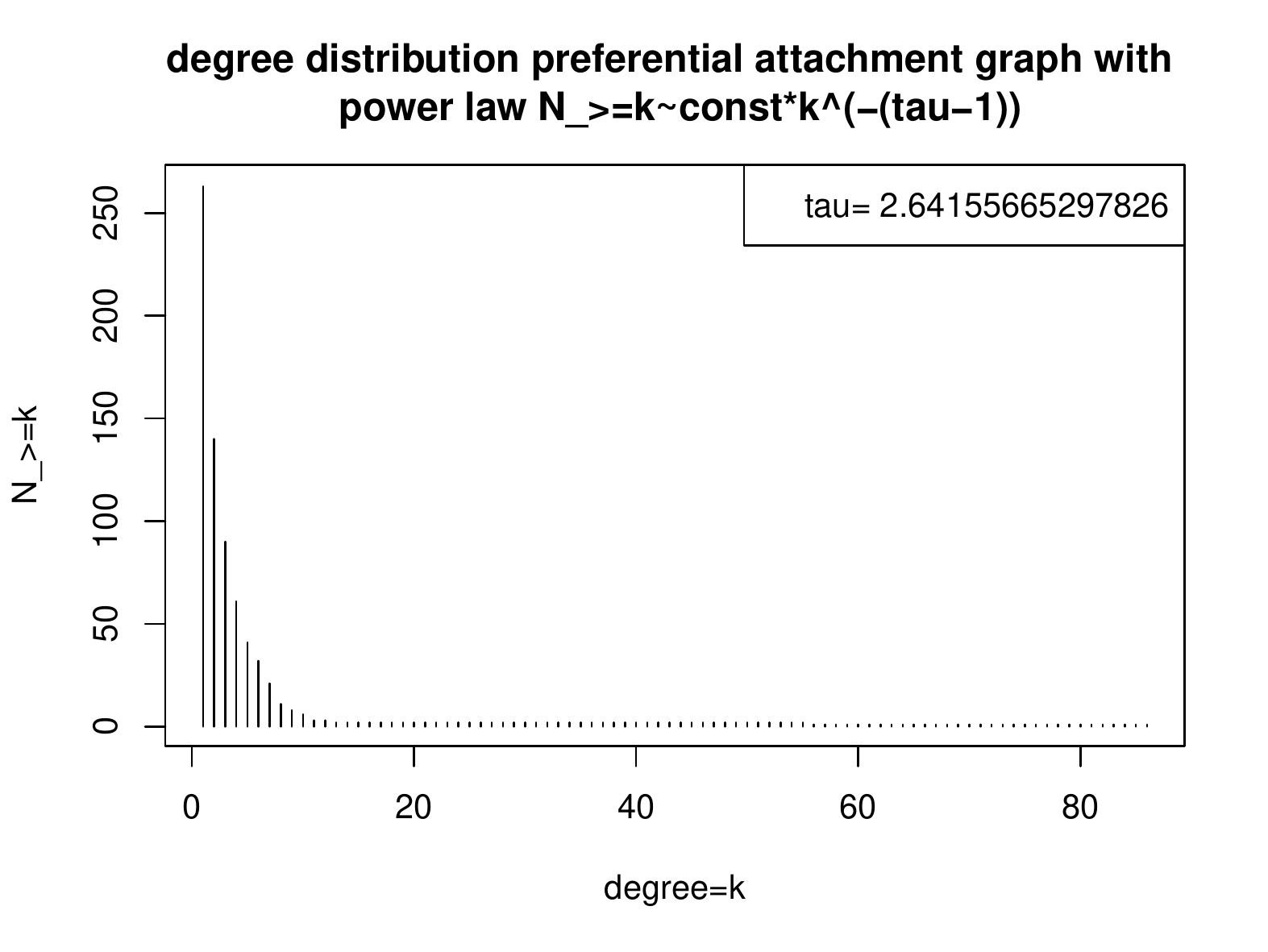}
\includegraphics[width=6cm]{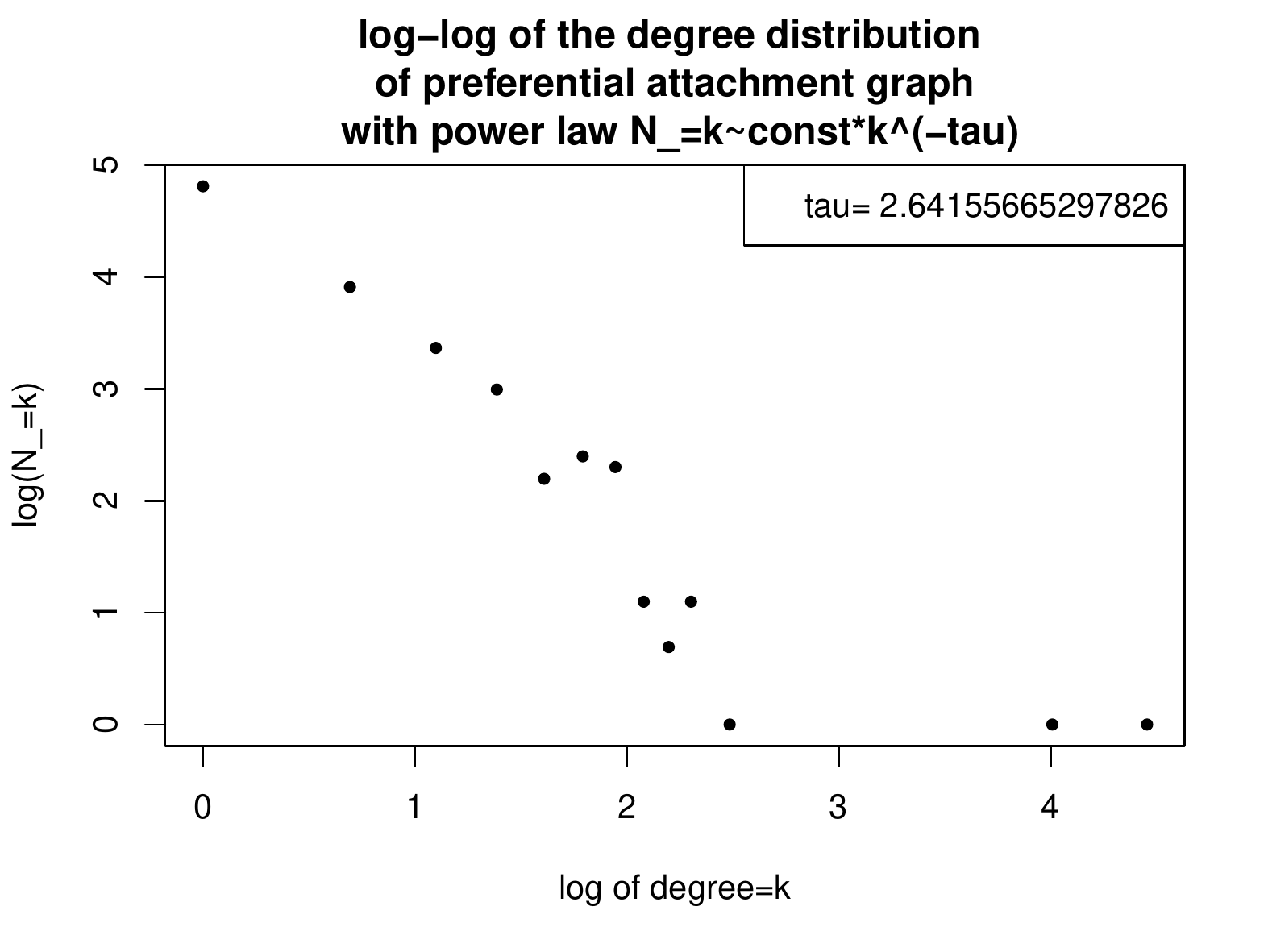}
\caption{Degree distribution of the protein network with $n=263$ nodes and $\tau=2.642$.}\label{deg_Prot}
\end{figure}
The question that arises is if we can model the protein network with our \PA \ model. As we can see in figure \ref{ProtNW} the graph is not connected and this actually is the case for the majority of discovered protein families. Yet by construction every graph coming from our \PA \ model is connected. But since our network only contains discovered proteins it still could be possible to model it with our \PA \ model if we had the full network with all connections. Lets therefore look at the average degree in our network. It is given by
\[\frac{1}{n}\sum_{i=1}^n D(i)\approx 3.027,\]
which is not possible with our model since by equation (\ref{mean}) we always get an even degree. As stated before we are only looking at a subgraph of the network of all existing proteins of this family. Therefore it could still be possible to get the protein network as a subgraph of a \PA \ graph if the parameters are chosen carefully. Also one could think of a \PA \ model which does not yield a connected graph.


\section{Outlook}\label{outlook}
The discussion in section \ref{proteins} did not confirm that it is possible to model protein networks with our \PA \ model. But like stated before there might be a solution to this by using a different \PA \ approach which yields an unconnected graph with a random number of edges added every time a new node enters the graph.\\
Most protein networks are scale free but there are actually many possibilities to construct scale free graphs with different models such as the SN-Model introduced by Frisco in 2011 \cite{Frisco} where the nodes itself have a structure and are connected according to the differences in these structures. It is also possible to imitate the construction of proteins by using a Hidden Markov Model where the overlaying graph corresponds to the protein network.

\newpage
\bibliographystyle{abbrv}
\bibliography{literature}

\emph{Institute of Stochastics and Applications, University of Stuttgart, Pfaffenwaldring 57, 70569 Stuttgart, Germany}\\
\emph{E-mail adress:} \texttt{klemens.taglieber@mathematik.uni-stuttgart.de}\\
\\
\emph{Institute of Stochastics and Applications, University of Stuttgart, Pfaffenwaldring 57, 70569 Stuttgart, Germany}\\
\emph{E-mail adress:} \texttt{uta.freiberg@mathematik.uni-stuttgart.de}
\end{document}